\documentclass[12pt]{amsart}

\input xy
\xyoption{all}

\usepackage{geometry}
\usepackage{amsmath}
\usepackage{amssymb}
\usepackage{cite}
\usepackage{amsthm}  
\usepackage{mathtools}
\usepackage{amsmath, amsthm, amssymb}
\usepackage{listings}
\usepackage{url}
\usepackage{hyperref}
\usepackage[utf8]{inputenc}
\usepackage{graphicx}

\theoremstyle{definition}
\newtheorem{mydef}{Definition}[section]
\newtheorem{lem}[mydef]{Lemma}
\newtheorem{thm}[mydef]{Theorem}
\newtheorem{cor}[mydef]{Corollary}

\newtheorem{hypothesis}[mydef]{Hypothesis}
\newtheorem{prop}[mydef]{Proposition}

\newtheorem{remark}[mydef]{Remark}

\newtheorem{fact}[mydef]{Fact}

\newcommand{\fct}[2]{{}^{#1}#2}

\newcommand{\bigN}{\widehat{N}}

\newcommand{\rest}{\upharpoonright}

\newcommand{\s}{\mathfrak{s}}

\newbox\noforkbox \newdimen\forklinewidth
\forklinewidth=0.3pt \setbox0\hbox{$\textstyle\smile$}
\setbox1\hbox to \wd0{\hfil\vrule width \forklinewidth depth-2pt
 height 10pt \hfil}
\wd1=0 cm \setbox\noforkbox\hbox{\lower 2pt\box1\lower
2pt\box0\relax}
\def\unionstick{\mathop{\copy\noforkbox}\limits}
\def\nf{\unionstick}

\newcommand{\nfs}[4]{#2 \nf_{#1}^{#4} #3}

\setbox0\hbox{$\textstyle\smile$}
\setbox1\hbox to \wd0{\hfil{\sl /\/}\hfil} \setbox2\hbox to
\wd0{\hfil\vrule height 10pt depth -2pt width
               \forklinewidth\hfil}
\newbox\doesforkbox
\setbox\doesforkbox\hbox{\lower 0pt\box1 \lower
2pt\box2\lower2pt\box0\relax}

\def\1nf{\unionstick^{(1)}}

\def\2nf{\unionstick^{(2)}}

\def\3nf{\unionstick^{(3)}}

\def\lta{<}
\def\lea{\le}
\def\gta{>}
\def\gea{\ge}

\def\ltu{\lta_{\text{univ}}}

\def\gtu{\gta_{\text{univ}}}

\newcommand{\ltl}[2]{\lta_{#1,#2}}

\newcommand{\gtl}[2]{\gta_{#1,#2}}

\def\ltg{\vartriangleleft}
\def\gtg{\vartriangleright}

\newcommand{\hanf}[1]{h (#1)}
\newcommand{\hanfn}[1]{h_{#1}}

\newcommand{\goodm}{\text{good}^-}
\newcommand{\goodms}[1]{\text{good}^{-#1}}

\def\forkindep{\mathrel{\raise0.2ex\hbox{\ooalign{\hidewidth$\vert$\hidewidth\cr\raise-0.9ex\hbox{$\smile$}}}}}

\title{Forking and superstability in tame AECs}
\date{\today} 

\author{Sebastien Vasey}
\thanks{The author is supported by the Swiss National Science Foundation.}
\email{sebv@cmu.edu}
\urladdr{http://math.cmu.edu/\textasciitilde svasey/}
\address{Department of Mathematical Sciences, Carnegie Mellon University, Pittsburgh, Pennsylvania, USA}

\begin{document}

\begin{abstract}
We prove that any tame abstract elementary class categorical in a suitable cardinal has an eventually global good frame: a forking-like notion defined on all types of single elements. This gives the first known general construction of a good frame in ZFC. We show that we already obtain a well-behaved independence relation assuming only a superstability-like hypothesis instead of categoricity. These methods are applied to obtain an upward stability transfer theorem from categoricity and tameness, as well as new conditions for uniqueness of limit models.
\end{abstract}

\maketitle

\tableofcontents

\section{Introduction}

In 2009, Shelah published a two volume book \cite{shelahaecbook, shelahaecbook2} on classification theory for abstract elementary classes. The central new structural notion is that of a good $\lambda$-frame (for a given abstract elementary class (AEC) $K$): a generalization of first-order forking to types over models of size $\lambda$ in $K$ (see Section \ref{good-frames-subsec} below for the precise definition). The existence of a good frame shows that $K$ is very well-behaved at $\lambda$ and the aim was to use this frame to deduce more on the structure of $K$ above $\lambda$. Part of this program has already been accomplished through several hundreds of pages of hard work (see for example \cite{sh576}, \cite[Chapter 2 and 3]{shelahaecbook}, \cite{jasi,jrsh875,jash940-v1, jarden-prime}). Among many other results, Shelah shows that good frames exist under strong categoricity assumptions and additional set-theoretic hypotheses:

\begin{fact}[Theorem II.3.7 in \cite{shelahaecbook}]\label{good-frame-existence}
  Assume $2^{\lambda} < 2^{\lambda^+} < 2^{\lambda^{++}}$ and the weak diamond ideal in $\lambda^+$ is not $\lambda^{++}$-saturated.

  Let $K$ be an abstract elementary class with $\text{LS} (K) \le \lambda$. Assume:

  \begin{enumerate}
    \item $K$ is categorical in $\lambda$ and $\lambda^+$.
    \item $0 < I(\lambda^{++}, K) < \mu_{\text{unif}} (\lambda^{++}, 2^{\lambda^{+}})$
  \end{enumerate}

  Then $K$ has a good $\lambda^+$-frame.
\end{fact} 

It is a major open problem whether the set-theoretic hypotheses in Fact \ref{good-frame-existence} are necessary. In this paper, we show that if the class already has some \emph{global} structure, then good frames are much easier to build. For example we prove, \emph{in ZFC} (see Theorem \ref{main-thm-categ}):

\begin{thm}\label{intro-categ-succ}
  Let $K$ be an abstract elementary class with amalgamation and no maximal models. Assume $K$ is categorical in a high-enough\footnote{In fact, $\lambda$ can be taken to be above $\hanf{\hanf{\hanf{\text{LS} (K)}}^+}$, where $\hanf{\mu} = \beth_{(2^{\mu})^+}$.} successor $\lambda^+$. Then $K$ has a type-full good $\lambda$-frame.
\end{thm}

By the main theorem of \cite{sh394}, the hypotheses of Theorem \ref{intro-categ-succ} imply $K$ is categorical in $\lambda$. On the other hand, we do not need any set-theoretic hypothesis and we do not need to know anything about the number of models in $\lambda^{++}$. Moreover, the frame Shelah constructs typically defines a notion of forking only for a restricted class of basic types (the minimal types). With a lot of effort, he then manages to show \cite[Section III.9]{shelahaecbook} that under some set-theoretic hypotheses one can always extend a frame to be type-full. In our frame, forking is directly defined for every type. This is technically very convenient and closer to the first-order intuition. Of course, we pay for this luxury by assuming amalgamation and no maximal models\footnote{After submitting this paper, we discovered that Shelah claims to build a good frame in ZFC from categoricity in a high-enough cardinal in Chapter IV of \cite{shelahaecbook}. We were unable to fully check Shelah's proof. At the very least, our construction using tameness is simpler and gives much lower Hanf numbers.}.

Our proof relies on two key properties of AECs. The first one is tameness (a locality property of Galois types, see Definition \ref{tameness-def}), and assuming it lets us relax the ``high-enough successor'' assumption in Theorem \ref{intro-categ-succ}, see Theorem \ref{main-thm-technical-categ}:

\begin{thm}\label{intro-categ-limit}
  Let $K$ be an abstract elementary class with amalgamation and no maximal models. Assume $K$ is $\mu$-tame and categorical in some cardinal $\lambda$ such that $\text{cf} (\lambda) > \mu$. Then $K$ has a type-full good $\ge \lambda$-frame. 
\end{thm}

That is, not only do we obtain a good $\lambda$-frame, but we can also extend this frame to any model of size $\ge \lambda$ (this last step essentially follows from earlier work of Boney \cite{ext-frame-jml}). Hence we obtain a global forking notion above $\lambda$, although only defined for 1-types. A forking notion for types of all lengths is obtained in \cite{bg-v7} (using stronger tameness hypotheses than ours) but the authors assume the extension property for coheir, and it is unclear when this holds, even assuming categoricity everywhere. Thus our result partially answers \cite[Question 7.1]{bgkv-v2} (which asked when categoricity together with tameness implies the existence of a forking-like notion for types of all lengths satisfying uniqueness, local character, and extension). We also obtain new theorems whose statements do not mention frames: 

\begin{cor}\label{stab-transfer}
  Let $K$ be an abstract elementary class with amalgamation and no maximal models. Assume $K$ is $\mu$-tame and categorical in some cardinal $\lambda$ such that $\text{cf} (\lambda) > \mu$. Then $K$ is stable everywhere.
\end{cor}
\begin{remark}
  Shelah already established in \cite{sh394} that categoricity in $\lambda > \text{LS} (K)$ implies stability below $\lambda$ (assuming amalgamation and no maximal models). The first upward stability transfer for tame AECs appeared in \cite{tamenessone}. Later, \cite{b-k-vd-spectrum} gave some variations, showing for example $\aleph_0$-stability and a strong form of tameness implies stability everywhere. Our upward stability transfer improves on \cite[Corollary 4.7]{b-k-vd-spectrum} which showed that categoricity in a successor $\lambda$ implies stability in $\lambda$.
\end{remark}
\begin{cor}
  Let $K$ be an abstract elementary class with amalgamation and no maximal models. Assume $K$ is $\mu$-tame and categorical in some cardinal $\lambda$ such that $\text{cf} (\lambda) > \mu$. Then $K$ has a unique limit model\footnote{This holds even in the stronger sense of \cite[Theorem 3.3.7]{shvi635}, i.e.\ two limit models over the same base are isomorphic over the base.} in every $\lambda' \ge \lambda$.
\end{cor}
\begin{remark}
  This is also new and complements the conditions for uniqueness of limit models given in \cite{sh394}, \cite{vandierennomax}, and \cite{gvv-v3}.
\end{remark}

The second key property in our proof is a technical condition we call local character of $\mu$-splitting for $\ltg$-chains (see Definition \ref{lc-ns-def}). This follows from categoricity in a cardinal of cofinality larger than $\mu$ and we believe it is a good candidate for a definition of superstability, at least in the tame context. Under this hypothesis, we already obtain a forking notion that is well-behaved for $\mu^+$-saturated base models and can prove the upward stability transfer given by Corollary \ref{stab-transfer}. Local character of splitting already played a key role in other papers such as \cite{shvi635}, \cite{vandierennomax}, and \cite{gvv-v3}. 

Even if this notion of superstability fails to hold, we can still look at the length of the chains for which $\mu$-splitting has local character (analogous to the cardinal $\kappa (T)$ in the first-order context). Using GCH, we can generalize one direction of the first-order characterization of the stability spectrum (Theorem \ref{main-thm-spectrum-2}).

The paper is structured as follows: In Section \ref{prelim}, we review background in the theory of AECs and give the definition of good frames. In Section \ref{poor-man-frame}, we fix a cardinal $\mu$ and build a $\mu$-frame-like object named a \emph{skeletal frame}. This is done using the weak extension and uniqueness properties of splitting isolated by VanDieren \cite{vandierenthesis}, together with the assumption of local character of splitting. In Section \ref{going-up-sec}, we show that some of the properties of our skeletal frame in $\mu$ lift to cardinals above $\mu$ (and in fact become better than they were in $\mu$). This is done using the same methods as in \cite[Section II.2]{shelahaecbook}. 

In Section \ref{good-frame-without-symmetry}, we show assuming tameness that the other properties of the skeletal frame lift as well and similarly become better, so that we obtain (if we restrict ourselves to $\mu^+$-saturated models and so, assuming categoricity in the right cardinal, to all models) all the properties of a good frame except perhaps symmetry. This uses the ideas from \cite{ext-frame-jml}. Next in Section \ref{getting-sym} we show how to get symmetry by using more tameness together with the order property (this is where we really use that we have structure properties holding globally and not only at a few cardinals). Finally, we put everything together in Section \ref{main-thm}. In Section \ref{conclusion}, we conclude.

At the beginning of Sections \ref{poor-man-frame}, \ref{going-up-sec}, \ref{good-frame-without-symmetry}, and \ref{getting-sym}, we give hypotheses that are assumed to hold everywhere in those sections. We made an effort to show clearly how much of the structural properties (amalgamation, tameness, superstability, etc.) are used at each step, but our construction is new even for the case of a totally categorical AEC $K$ with amalgamation, no maximal models, and $\text{LS} (K)$-tameness. It might help the reader to keep this case in mind throughout.

This paper was written while working on a Ph.D.\ thesis under the direction of Rami
Grossberg at Carnegie Mellon University and I would like to thank Professor Grossberg for his guidance and assistance in my research in general and in this work
specifically. I also thank John T. Baldwin, Will Boney, Adi Jarden, Alexei Kolesnikov, and the anonymous referee for valuable comments that helped improve the presentation of this paper.

\section{Preliminaries}\label{prelim}

\subsection{Abstract elementary classes}

We assume the reader is familiar with the definition of an abstract elementary class (AEC) and the basic related concepts. See \cite{grossberg2002} for an introduction. 

For the rest of this section, fix an AEC $K$. We denote the partial ordering on $K$ by $\lea$, and write $M \lta N$ if $M \lea N$ and $M \neq N$. For $R$ a binary relation on $K$ and $\delta$ an ordinal, an \emph{$R$-increasing chain} $(M_i)_{i < \delta}$ is a sequence of models in $K$ such that for all $i < \delta$, if $i + 1 < \delta$ then $R(M_i, M_{i + 1})$. The chain is \emph{continuous} if it is $\lea$-increasing and for any limit $i < \delta$, $M_i = \bigcup_{j < i} M_j$. When we talk of an increasing chain, we mean a $\lea$-increasing chain. Strictly increasing means $\lta$-increasing. 

For $K$ an abstract elementary class and $\mathcal{F}$ an interval\footnote{The definitions that follow make sense for an arbitrary set of cardinals $\mathcal{F}$, but the proofs of most of the facts below require that $\mathcal{F}$ is an interval.} of cardinals of the form $[\lambda, \theta)$, where $\theta > \lambda$ is either a cardinal or $\infty$, let $K_{\mathcal{F}} := \{M \in K \mid \|M\| \in \mathcal{F}\}$. We write $K_\lambda$ instead of $K_{\{\lambda\}}$, $K_{\ge \lambda}$ instead of $K_{[\lambda, \infty)}$ and $K_{\le \lambda}$ instead of $K_{[0, \lambda]}$.

The following properties of AECs are classical:

\begin{mydef}[Amalgamation, joint embedding, no maximal models] 
  Let $\mathcal{F}$ be an interval of cardinals as above.
  \begin{enumerate}
    \item $K_{\mathcal{F}}$ has \emph{amalgamation} if for any $M_0 \lea M_\ell \in K_{\mathcal{F}}$, $\ell = 1,2$ there exists $N \in K_{\mathcal{F}}$ and $f_\ell : M_\ell \xrightarrow[M_0]{} N$, $\ell = 1,2$.
    \item $K_{\mathcal{F}}$ has \emph{joint embedding} if for any $M_\ell \in K_{\mathcal{F}}$, $\ell = 1,2$ there exists $N \in K_{\mathcal{F}}$ and $f_\ell : M_\ell \rightarrow N$, $\ell = 1,2$.
    \item $K_{\mathcal{F}}$ has \emph{no maximal models} if for any $M \in K_{\mathcal{F}}$ there exists $N \gta M$ in $K_{\mathcal{F}}$.
  \end{enumerate}
\end{mydef}
\begin{fact}
  Let $\mathcal{F}$ be an interval of cardinals as above. 

  \begin{enumerate}
    \item If $K_\mu$ has no maximal models for all $\mu \in \mathcal{F}$, then $K_{\mathcal{F}}$ has no maximal models.
    \item If $K_\mu$ has amalgamation for all $\mu \in \mathcal{F}$, then $K_{\mathcal{F}}$ has amalgamation. 
  \end{enumerate}
\end{fact}
\begin{proof}
  No maximal models is straightforward and amalgamation is \cite[Conclusion I.2.12]{shelahaecbook}.
\end{proof}

Finally, we will also use:

\begin{lem}\label{jep-from-amalgamation}
  Let $\mathcal{F} = [\lambda, \theta)$ be an interval of cardinals as above.
  If $K_{\mathcal{F}}$ has amalgamation and $K_\lambda$ has joint embedding, then $K_{\mathcal{F}}$ has joint embedding.
\end{lem}
\begin{proof}[Proof sketch]
  Let $M_\ell \in K_{\mathcal{F}}$, $\ell = 1,2$. Pick $M_\ell' \lea M_\ell$ of size $\lambda$. Use joint embedding on $M_1', M_2'$, then use amalgamation.
\end{proof}

\subsection{Galois types, stability, and tameness}

We assume familiarity with Galois types (see \cite[Section 6]{grossberg2002}). For $M \in K$, we write $S^\alpha (M)$ for the set of Galois types of sequences of length $\alpha$ over $M$. We will at one point also consider types over the empty set, which are defined analogously (see e.g.\ \cite[Definition 1.4]{sh394-updated}). We write $S (M)$ for $S^1 (M)$. We write $S^{\text{na}} (M)$ for the set of \emph{nonalgebraic} 1-types over $M$, that is:

$$
S^{\text{na}} (M) := \{\text{gtp} (a / M; N) \mid a \in N \backslash M, M \lea N \in K\}
$$

From now on, we will write $\text{tp} (a / M; N)$ for $\text{gtp} (a / M; N)$.

We briefly review the notion of tameness. Although it appears implicitly (for saturated models) in \cite{sh394}, tameness as a property of AECs was first introduced in \cite{tamenessone} and used to prove a stability spectrum theorem. It was later used in \cite{tamenessthree} to prove an upward categoricity transfer. Our definition follows \cite[Definition 3.1]{tamelc-jsl}.

\begin{mydef}[Tameness]\label{tameness-def}
  Let $\lambda > \kappa \ge \text{LS} (K)$. Let $\alpha$ be a cardinal. We say that $K$ is \emph{$(\kappa, \lambda)$-tame for $\alpha$-length types} if for any $M \in K_{\le \lambda}$ and any $p, q \in S^\alpha (M)$, if $p \neq q$, then there exists $M_0 \in K_{\le \kappa}$ with $M_0 \lea M$ such that $p \upharpoonright M_0 \neq q \upharpoonright M_0$. We define similarly $(\kappa, <\lambda)$-tame, $(<\kappa, \lambda)$-tame, etc. When $\lambda = \infty$, we omit it. When $\alpha = 1$, we omit it. We say that $K$ is \emph{fully $\kappa$-tame} if it is $\kappa$-tame for all lengths.
\end{mydef}

We also recall that we can define a notion of stability:

\begin{mydef}[Stability]
  Let $\lambda \ge \text{LS} (K)$ and $\alpha$ be cardinals. We say that $K$ is \emph{$\alpha$-stable in $\lambda$} if for any $M \in K_\lambda$, $|S^\alpha (M)| \le \lambda$.

  We say that $K$ is \emph{stable} in $\lambda$ if it is $1$-stable in $\lambda$.

  We say that $K$ is \emph{$\alpha$-stable} if it is $\alpha$-stable in $\lambda$ for some $\lambda \ge \text{LS} (K)$. We say that $K$ is \emph{stable} if it is $1$-stable in $\lambda$ for some $\lambda \ge \text{LS} (K)$. We write ``unstable'' instead of ``not stable''. 

  We define similarly stability for $K_{\mathcal{F}}$, e.g.\ $K_{\mathcal{F}}$ is stable if and only if $K$ is stable in $\lambda$ for some $\lambda \in \mathcal{F}$.
\end{mydef}

\begin{remark}\label{monot-stab}
  If $\alpha < \beta$, and $K$ is $\beta$-stable in $\lambda$, then $K$ is $\alpha$-stable in $\lambda$.
\end{remark}

The following follows from \cite[Theorem 1.1]{longtypes-v2}.

\begin{fact}\label{stab-longtypes}
  Let $\lambda \ge \text{LS} (K)$. Let $\alpha$ be a cardinal. Assume $K$ is stable in $\lambda$ and $\lambda^\alpha = \lambda$. Then $K$ is $\alpha$-stable in $\lambda$.
\end{fact}

\subsection{Universal and limit extensions}

\begin{mydef}[Universal and limit extensions]
  For $M, N \in K$, we say that $N$ is \emph{universal over $M$} (written $M \ltu N$) if and only if $M < N$ and for any $M' \in K_{\|M\|}$ with $M' \gea M$, $M'$ can be embedded inside $N$ over $M$. We also write $N \gtu M$ for $M \ltu N$.

  For $\mu \ge \text{LS} (K)$ and $0 < \delta < \mu^+$ an ordinal, we say that $N$ is \emph{($\mu$,$\delta$)-limit} over $M$ (written $M \ltl{\mu}{\delta} N$) if and only if $M, N \in K_\mu$, $M \lea N$, and there is a $\ltu$-increasing chain $(M_i)_{i \le \delta}$ with $M_0 = M$, $M_\delta = N$ and $M_\delta = \bigcup_{i < \delta} M_i$ if $\delta$ is limit. We also write $N \gtl{\mu}{\delta} M$ for $M \ltl{\mu}{\delta} N$.

  We say that a model $N$ is \emph{limit} if it is $(\|N\|, \gamma)$-limit over $M$ for some $M \lea N$ and some limit ordinal $\gamma < \mu^+$.
\end{mydef}
\begin{mydef}
  A model $N \in K$ is \emph{$\mu$-model-homogeneous} if for any $M \lea N$ with $\|M\| < \mu$, we have $M \ltu N$. $N$ is \emph{model-homogeneous} if it is $\|N\|$-model-homogeneous.
\end{mydef}
\begin{fact}\label{gtu-existence}
  Let $\mu \ge \text{LS} (K)$. Assume $K_\mu$ has amalgamation, no maximal models, and is stable. For any $M \in K_\mu$, there exists $N \in K_\mu$ such that $M \ltu N$. Therefore there is a model-homogeneous $N \in K_{\mu^+}$ with $M \lta N$.
\end{fact}
\begin{proof}
  The first part is by \cite[Claim II.1.16.1(a)]{shelahaecbook}. The second part follows from iterating the first part $\mu^+$ many times.
\end{proof}
\begin{remark}
  By \cite[Lemma 0.26]{sh576}, for $\mu > \text{LS} (K)$, $N$ is $\mu$-model-homogeneous if and only if it is $\mu$-saturated.
\end{remark}

The next proposition is folklore and the results appear in several places in the literature (see for example \cite[Lemma 2.2]{sh394}). For the convenience of the reader, we have included the proofs.

\begin{prop}\label{ltl-basic-props}
 Let $M_0, M_1, M_2 \in K_\mu$, $\mu \ge \text{LS} (K)$ and $0 < \delta < \mu^+$. Then:
  \begin{enumerate}
  \item $M_0 \ltl{\mu}{\delta} M_1$ implies $M_0 \ltu M_1$.
  \item \label{univ-trans} $M_0 \ltu M_1 \lea M_2$ implies $M_0 \ltu M_2$.
  \item Assume $K_\mu$ has amalgamation. Then  $M_0 \lea M_1 \ltl{\mu}{\delta} M_2$ implies $M_0 \ltl{\mu}{\delta} M_2$.
  \item Assume $K_\mu$ has amalgamation, no maximal models, and is stable. Then there exists $M_0'$ such that $M_0 \ltl{\mu}{\delta} M_0'$.
  \item Conversely, if for every $M_0 \in K_\mu$ there exists $M_0' \in K_\mu$ such that $M_0 \ltu M_0'$, then $K_\mu$ has amalgamation, no maximal models, and is stable.
  \end{enumerate}
\end{prop}
\begin{proof} \
  \begin{enumerate}
    \item Fix $(N_i)_{i \le \delta}$ witnessing that $M_0 \ltl{\mu}{\delta} M_1$. Let $M_0' \gea M_0$ have size $\mu$. Since $\delta > 0$, $N_1$ is well defined, and is universal over $N_0 = M_0$, hence $M_0'$ can be embedded inside $N_1$ over $M_0$, and hence since $N_1 \lea M_1$ can be embedded inside $M_1$ over $M_0$.
    \item Let $M_0' \gea M_0$ have size $\mu$. Since $M_0'$ embed inside $M_1$ over $M_0$, it also embeds inside $M_2$ over $M_0$.
    \item Let $(N_i)_{i \le \delta}$ witness $M_0 \ltl{\mu}{\delta} M_1$. We show that $M_0 \ltu N_1$. This is enough since then $M_0 \frown (N_i)_{0 < i \le \delta}$ will witness that $M_0 \ltl{\mu}{\delta} M_2$. Let $M_0' \gea M_0$ have size $\mu$. By amalgamation, find $M_1' \gea M_1$ and $h: M_0' \xrightarrow[M_0]{} M_1'$. Now use universality of $M_2$ over $M_1$ to find $g: M_1' \xrightarrow[M_1]{} M_2$. Let $f := g \circ h$. Then $f: M_0' \xrightarrow[M_0]{} M_2$, as desired. 
    \item Iterate Fact \ref{gtu-existence} $\delta$ many times.
    \item Let $M_0 \in K_\mu$ and let $M_0' \gtu M_0$ be in $K_\mu$. $M_0'$ witnesses that $M_0$ is not maximal in $K_\mu$. Moreover, $M_0$ is an amalgamation base, since any two models of size $\mu$ extending $M_0$ can amalgamated over $M_0$ inside $M_0'$. Finally, all types over $M_0$ are realized in $M_0'$ which has size $\mu$, there can be at most $\mu$ many of them, so stability follows.
  \end{enumerate}
\end{proof}

We give orderings satisfying the conclusion of Proposition \ref{ltl-basic-props} a name:

\begin{mydef}[Abstract universal ordering]
  An \emph{abstract universal ordering on $K_\mu$} is a binary relation $\ltg$ on $K_\mu$ satisfying the following properties. For any $M_0, M_1, M_2 \in K_\mu$:

  \begin{enumerate}
  \item $M_0 \ltg M_1$ implies $M_0 \ltu M_1$.
  \item There exists $N_0 \in K_\mu$ such that $M_0 \ltg N_0$.
  \item $M_0 \lea M_1 \ltg M_2$ implies $M_0 \ltg M_2$.
  \item Closure under isomorphism: if $M_0 \ltg M_1$ and $f: M_1 \cong M_1'$, then $f[M_0] \ltg M_1'$.
  \end{enumerate}

  Note that this implies that $\ltg$ is a strict partial ordering on $K_\mu$ extending $\lta$.

  For $0 < \delta < \mu^+$, a model $M \in K_\mu$ is \emph{$(\delta, \ltg)$-limit} if there exists a $\ltg$-increasing chain $(M_i)_{i < \delta}$ in $K_\mu$ such that $M = \bigcup_{i < \delta} M_i$. $M$ is \emph{$\ltg$-limit} if there exists a limit $\delta$ such that $M$ is $(\delta, \ltg)$-limit.
\end{mydef}

\begin{remark}\label{rmk-good-ordering-stable}
Assume $K_\mu$ has amalgamation, no maximal models, and is stable. Then by Proposition \ref{ltl-basic-props}, for any $0 < \delta < \mu^+$, $\ltl{\mu}{\delta}$ is an abstract universal ordering on $K_\mu$. Moreover, the existence of \emph{any} abstract universal ordering on $K_\mu$ implies that $\ltu$ is an abstract universal ordering, and hence that $K_\mu$ has amalgamation, no maximal models, and is stable.
\end{remark}

Let $\text{LS} (K) \le \mu < \lambda$. Even assuming stability everywhere, is is unclear whether there should be any model-homogeneous model in $\lambda$ (think for example of the case $\text{cf} (\lambda) = \omega$). The following tells us that we can at least get an approximation to one: we can do the usual construction of special models in a cardinal $\lambda$ if $K$ is stable below $\lambda$. This will be used in the proof of the superstability theorem (Theorem \ref{lambdap-stability}).

\begin{lem}\label{univ-embed}
  Let $\text{LS} (K) \le \mu^+ < \lambda$. Assume $K_{[\mu, \lambda)}$ has amalgamation, no maximal models, and is stable in $\mu'$ for unboundedly many $\mu < \mu' < \lambda$ (that is, for any $\mu < \mu' < \lambda$, there exists $\mu' \le \mu'' < \lambda$ such that $K_{\mu''}$ is stable).
    
    For any $N_0 \in K_{[\mu, \lambda)}$, there exists $(N_i)_{i < \lambda}$ $\ltu$-increasing continuous in $K_{[\mu, \lambda)}$ with each $N_{i + 1}$ $\mu^+$-model-homogeneous. Moreover any $M \in K_{[\mu, \lambda]}$ such that $N_0 \lea M$ can be embedded inside $N := \bigcup_{i < \lambda} N_i$ over $N_0$.
\end{lem}

\begin{proof}
  We build $(N_i)_{i < \lambda}$ by induction. $N_0$ is already given and without loss of generality $\|N_0\| \ge \mu^+$. Take unions at limits and for a given $N_i$, first take $N_i' \ge N_i$ such that $K_{\|N_i'\|}$ is stable, and iterate Fact \ref{gtu-existence} $\mu^+$-many times to pick $N_{i + 1} \in K_{\|N_i'\|}$ which is also $\mu^+$-model-homogeneous such that $N_i' \ltu N_{i + 1}$ (and so by Proposition \ref{ltl-basic-props} also $N_i \ltu N_{i + 1}$).

  Now given $M \in K_{[\mu, \lambda]}$ with $N_0 \lea M$, let $(M_i)_{i \le \lambda}$ be an increasing continuous resolution of $M$ such that $\|M_i\| < \lambda$ for all $i < \lambda$ and $M_0 = N_0$. Inductively build $(f_i)_{i \le \lambda}$ an increasing continuous chain of $K$-embeddings such that for each $i \le \lambda$, $f_i : M_i \xrightarrow[M_0]{} N_i$. This is easy since $N_{i + 1} \gtu N_i$ for all $i < \lambda$. Then $f_\lambda$ embeds $M$ into $N$.
\end{proof}

\subsection{Good frames}\label{good-frames-subsec}

Good frames were first defined in \cite[Chapter II]{shelahaecbook}. The idea is to provide a localized (i.e.\ only for base models of a given size $\lambda$) axiomatization of a forking-like notion for (a ``nice enough'' set of) 1-types. Jarden and Shelah (in \cite{jrsh875}) later gave a slightly more general definition, not assuming the existence of a superlimit model and dropping some of the redundant clauses. We will use a slight variation here: we assume the models come from $K_{\mathcal{F}}$, for $\mathcal{F}$ an interval, instead of just $K_\lambda$. We first adapt the definition of a pre-$\lambda$-frame from \cite[Definition III.0.2.1]{shelahaecbook} to such an interval:

\begin{mydef}[Pre-frame]
  Let $\mathcal{F}$ be an interval of the form $[\lambda, \theta)$, where $\lambda$ is a cardinal, and $\theta > \lambda$ is either a cardinal or $\infty$.

  A \emph{pre-$\mathcal{F}$-frame} is a triple $\mathfrak{s} = (K, \nf, S^{\text{bs}})$, where:

  \begin{enumerate}
  \item $K$ is an abstract elementary class\footnote{In \cite[Definition III.0.2.1]{shelahaecbook}, Shelah only asks that $K$ contains the models of size $\mathcal{F}$ of an AEC. For easy of exposition, we do not adopt this approach.} with $\lambda \ge \text{LS} (K)$, $K_\lambda \neq \emptyset$.
  \item $S^{\text{bs}} \subseteq \bigcup_{M \in K_{\mathcal{F}}} S^{\text{na}} (M)$. For $M \in K_{\mathcal{F}}$, we write $S^{\text{bs}} (M)$ for $S^{\text{bs}} \cap S^{\text{na}} (M)$.
  \item $\nf$ is a relation on quadruples of the form $(M_0, M_1, a, N)$, where $M_0 \lea M_1 \lea N$, $a \in N$, and $M_0$, $M_1$, $N$ are all in $K_{\mathcal{F}}$. We write $\nf(M_0, M_1, a, N)$ or $\nfs{M_0}{a}{M_1}{N}$ instead of $(M_0, M_1, a, N) \in \nf$. 
  \item The following properties hold:
    \begin{enumerate}
      \item Invariance: If $f: N \cong N'$ and $\nfs{M_0}{a}{M_1}{N}$, then $\nfs{f[M_0]}{f(a)}{f[M_1]}{N'}$. If $\text{tp} (a / M_1; N) \in S^{\text{bs}} (M_1)$, then $\text{tp}(f (a) / f[M_1]; N') \in S^{\text{bs}} (f[M_1])$.
      \item Monotonicity: If $\nfs{M_0}{a}{M_1}{N}$, $M_0 \lea M_0' \lea M_1' \lea M_1 \lea N' \lea N \lea N''$ with $a \in N'$ and $N'' \in K_{\mathcal{F}}$, then $\nfs{M_0'}{a}{M_1'}{N'}$ and $\nfs{M_0'}{a}{M_1'}{N''}$.
      \item Nonforking types are basic: If $\nfs{M}{a}{M}{N}$, then $\text{tp} (a / M; N) \in S^{\text{bs}} (M)$.
    \end{enumerate}
  \end{enumerate}

  We write $\lambda$-frame instead of $\{\lambda\}$-frame, $(\ge \lambda)$-frame instead of $[\lambda, \infty)$-frame. We sometimes drop the $\mathcal{F}$ when it is clear from context.

    A pre-frame is \emph{type-full} if $S^{\text{bs}} (M) = S^{\text{na}} (M)$ for all $M \in K_{\mathcal{F}}$.

    For $\mathcal{F}' \subseteq \mathcal{F}$ an interval, we let $\mathfrak{s} \upharpoonright \mathcal{F}'$ denote the pre-$\mathcal{F}'$-frame defined in the obvious way by restricting the basic types and $\nf$ to models in $K_{\mathcal{F}'}$. For $\lambda' \in \mathcal{F}$, we write $\mathfrak{s} \upharpoonright \lambda'$ instead of $\mathfrak{s} \upharpoonright \{\lambda'\}$.
\end{mydef}

By the invariance and monotonicity properties, $\nf$ is really a relation on types. This justifies the next definition.

\begin{mydef}\label{s-forking}
  If $\mathfrak{s} = (K, \nf, S^{\text{bs}})$ is a pre-$\mathcal{F}$-frame, $p \in S (M_1)$ is a type, we say $p$ \emph{does not $\s$-fork over $M_0$} if $\nfs{M_0}{a}{M_1}{N}$ for some (equivalently any) $a$ and $N$ such that $p = \text{tp} (a / M_1; N)$.
\end{mydef}
\begin{remark}
  A pre-frame defines an \emph{abstract} notion of forking. That is, we only know that the relation $\nf$ satisfies some axioms but it could a-priori be defined arbitrarily. Later in the paper, we will study a specific definition of forking (based on splitting). While the specific definition we will give will coincide (over sufficiently saturated models) with first-order forking when the AEC is a class of models of a first-order theory, the reader should remember that we are working in much more generality than the first-order framework, hence most of the properties of first-order forking need not hold here.
\end{remark}
\begin{remark}
  We could have started from $(K, \nf)$ and defined the basic types as those that do not fork over their own domain. The existence property of good frames (see below) would then hold for free. Since we are sometimes interested in studying frames that only satisfy existence over a certain class of models (like the saturated models), we will not adopt this approach.
\end{remark}
\begin{remark}[Monotonicity of $\s$-forking]\label{monot-s-fork}
  If $\mathfrak{s} = (K, \nf, S^{\text{bs}})$ is a pre-$\mathcal{F}$-frame, $M_0 \lea M_1 \lea N_1 \lea N_0$ are in $K_{\mathcal{F}}$, and $p \in S^{\text{bs}} (N_0)$ does not $\mathfrak{s}$-fork over $M_0$, then by the monotonicity axiom, $p \rest N_1$ does not $\mathfrak{s}$-fork over $M_1$. We will use this fact freely.
\end{remark}




\begin{mydef}[Good frame]\label{good-frame-def}
  Let $\mathcal{F}$ be as above.

  A \emph{good $\mathcal{F}$-frame} is a pre-$\mathcal{F}$-frame $(K, \nf, S^{\text{bs}})$ satisfying in addition:

  \begin{enumerate}
  \item $K_{\mathcal{F}}$ has amalgamation, joint embedding, and no maximal model.
  \item \label{stab-cond} bs-Stability: $|S^{\text{bs}} (M)| \le \|M\|$ for all $M \in K_{\mathcal{F}}$.
  \item Density of basic types: If $M \lta N$ and $M, N \in K_{\mathcal{F}}$, then there is $a \in N$ such that $\text{tp} (a / M; N) \in S^{\text{bs}} (M)$.
  \item Existence: If $M \in K_{\mathcal{F}}$ and $p \in S^{\text{bs}} (M)$, then $p$ does not $\s$-fork over $M$.
  \item Extension: If $p \in S (N)$ does not $\s$-fork over $M$, and $N' \in K_{\mathcal{F}}$ is such that $N' \gea N$, then there is $q \in S (N')$ extending $p$ that does not $\s$-fork over $M$.
  \item Uniqueness: If $p, q \in S (N)$ do not $\s$-fork over $M$ and $p \upharpoonright M = q \upharpoonright M$, then $p = q$.
  \item Symmetry: If $\nfs{M_0}{a_1}{M_2}{N}$, $a_2 \in M_2$, and $\text{tp} (a_2 / M_0; N) \in S^{\text{bs}} (M_0)$, then there is $M_1$ containing $a_1$ and there is $N' \gea N$ such that $\nfs{M_0}{a_2}{M_1}{N'}$.
  \item Local character: If $\delta$ is a limit ordinal, $(M_i)_{i \le \delta}$ is an increasing chain in $K_{\mathcal{F}}$ with $M_\delta = \bigcup_{i < \delta} M_i$, and $p \in S^{\text{bs}} (M_\delta)$, then there exists $i < \delta$ such that $p$ does not $\s$-fork over $M_i$.
  \item Continuity: If $\delta$ is a limit ordinal, $(M_i)_{i \le \delta}$ is an increasing chain in $K_{\mathcal{F}}$ with $M_\delta = \bigcup_{i < \delta} M_i$, $p \in S (M_\delta)$ is so that $p \upharpoonright M_i$ does not $\s$-fork over $M_0$ for all $i < \delta$, then $p$ does not $\s$-fork over $M_0$.
  \item Transitivity\footnote{This actually follows from uniqueness and extension, see \cite[Claim II.2.18]{shelahaecbook}.}: If $M_0 \lea M_1 \lea M_2$, $p \in S (M_2)$ does not $\s$-fork over $M_1$ and $p \upharpoonright M_1$ does not $\s$-fork over $M_0$, then $p$ does not $\s$-fork over $M_0$.
  \end{enumerate}

  For $\mathbb{L}$ a list of properties, a $\goodms{\mathbb{L}}$ $\mathcal{F}$-frame is a pre-$\mathcal{F}$-frame that satisfies all the properties of good frames except possibly the ones in $\mathbb{L}$. In this paper, $\mathbb{L}$ will only contain symmetry and/or bs-stability. We abbreviate symmetry by $S$, bs-stability by $St$, and write $\goodm$ for $\goodms{(S, St)}$.

  We say that $K$ has a good $\mathcal{F}$-frame if there is a good $\mathcal{F}$-frame where $K$ is the underlying AEC (and similarly for $\goodm$).
\end{mydef}

\begin{remark}\label{frame-up}
  Using $\mathcal{F}$ instead of a single cardinal $\lambda$ is only a convenience: just like an abstract elementary class $K$ is determined by $K_{\text{LS} (K)}$, a $\goodm$ $\mathcal{F}$-frame $\mathfrak{s}$ is determined by $\mathfrak{s} \upharpoonright \lambda$, where $\lambda := \min (\mathcal{F})$. More precisely, if $\mathfrak{t}$ is a $\goodm$ $\mathcal{F}$-frame such that $\mathfrak{t} \upharpoonright \lambda = \mathfrak{s} \upharpoonright \lambda$, then the arguments from \cite[Section II.2]{shelahaecbook} show that $\mathfrak{t} = \mathfrak{s}$.
\end{remark}

Note that local character implies nonforking is always witnessed by a model of small size:

\begin{prop}\label{kappabar}
  Assume $\mathcal{F}$ is an interval of cardinals with minimum $\lambda$. Assume $\mathfrak{s} = (K, \nf, S^{\text{bs}})$ is a pre-$\mathcal{F}$-frame satisfying local character and transitivity. If $M \in K_{\mathcal{F}}$ and $p \in S^{\text{bs}} (M)$, then there exists $M' \in K_\lambda$ such that $p$ does not $\s$-fork over $M'$.
\end{prop}
\begin{proof}
  By induction on $\lambda' := \|M\|$. If $\lambda' = \lambda$, then since local character implies existence, we can take $M' := M$. Otherwise, $\lambda' > \lambda$ so we can take a resolution $(M_i)_{i < \lambda'}$ of $M$ such that $\lambda \le \|M_i\| < \lambda'$ for all $i < \lambda'$. By local character, there exists $i < \lambda'$ such that $p$ does not $\s$-fork over $M_i$. By monotonicity, $p \upharpoonright M_i$ does not $\s$-fork over $M_i$, so must be basic. By the induction hypothesis, there exists $M' \in K_\lambda$ such that $p \upharpoonright M_i$ does not $\s$-fork over $M'$. By transitivity, $p$ does not $\s$-fork over $M'$.
\end{proof}


\section{A skeletal frame from splitting}\label{poor-man-frame}

\begin{hypothesis} \
  \begin{enumerate}
    \item $K$ is an abstract elementary class. $\mu \ge \text{LS} (K)$ is a cardinal. $K_\mu \neq \emptyset$.
    \item $K_\mu$ has amalgamation.
  \end{enumerate}
\end{hypothesis}

In this section, we start our quest for a good frame. Note that we do \emph{not} assume that any abstract notion of forking is available to us at the start. Recall the following variations on first-order splitting from \cite[Definition 3.2]{sh394}:

\begin{mydef}
  For $p \in S (N)$, we say that $p$ \emph{$\mu$-splits over $M$} if $M \lea N$ and there exists $N_1, N_2 \in K_\mu$ so that $M \lea N_\ell \le N$ for $\ell = 1,2$, and $h: N_1 \cong_M N_2$ such that $h (p \upharpoonright N_1) \neq p \upharpoonright N_2$.

  When $\mu$ is clear from context, we drop it.
\end{mydef}

\begin{remark}[Monotonicity of splitting]
  If $p \in S (N)$ does not $\mu$-split over $M$ and $M \lea M' \lea N' \lea N$ are all in $K_\mu$, then $p \upharpoonright N'$ does not $\mu$-split over $M'$.
\end{remark}

\begin{remark}\label{nf-ns}
  If $\mathfrak{s}$ is a $\goodm$ $\mu$-frame, and $p$ does not $\s$-fork over $M$, then $p$ does not $\mu$-split over $M$ (this will not be used but follows from the uniqueness property, see e.g.\ \cite[Lemma 4.2]{bgkv-v2}). Thus splitting can be seen as a first approximation to a forking notion.
\end{remark}

Our starting point will be the following extension and uniqueness properties of splitting, first isolated by VanDieren \cite[Theorem II.7.9, Theorem II.7.11]{vandierenthesis}. Intuitively, they tell us that the usual uniqueness and extension property of a forking notion hold of splitting provided we have enough room (concretely, the base model has to be ``shifted'' by a universal extension).

\begin{fact}\label{ns-uq-ext}
Let $M_0 \ltu M \lea N$ with $M_0, M, N \in K_\mu$. Then:

\begin{enumerate}
  \item Weak uniqueness: If $p_\ell \in S (N)$ does not split over $M_0$, $\ell = 1,2$, $p_1 \upharpoonright M = p_2 \upharpoonright M$, then $p_1 = p_2$.
  \item Weak extension: If $p \in S (M)$ does not split over $M_0$, then there exists $q \in S (N)$ extending $p$ that does not split over $M_0$. Moreover, $q$ can be taken to be nonalgebraic if $p$ is nonalgebraic.
\end{enumerate}
\end{fact}
\begin{proof}
  See \cite[Theorem I.4.12]{vandierennomax} for weak uniqueness. For weak extension, use universality to get $h: N \xrightarrow[M_0]{} M$. Further extend $h$ to an isomorphism $\hat{h} : \widehat{N} \cong_{M_0} \widehat{M}$. So that $\widehat{M}$ contains a realization $a$ of $p$. Let $a' := \hat{h}^{-1} (a)$, and let $q := \text{tp} (a / N; \widehat{N})$. The proof of \cite[Theorem I.4.10]{vandierennomax} shows $q$ is indeed an extension of $p$ that does not split over $M_0$. In addition if $q$ is algebraic, $a' \in N$ so $a = h (a') \in M$, so $p$ is algebraic.
\end{proof}

We will mostly use those two properties instead of the exact definition of splitting. However, they characterize splitting in the following sense:

\begin{prop}
  Assume $K_\mu$ has amalgamation, no maximal models, and is stable. Let $\mathfrak{s}$ be a type-full pre-$\mu$-frame with underlying AEC $K$. The following are equivalent.

  \begin{enumerate}
    \item For all $M, N \in K_\mu$ with $M \lea N$ and all types $p \in S (N)$, if $p$ does not $\s$-fork over $M$, then for any $M \ltu M' \lea N$, $p$ does not split over $M'$.
    \item $\s$-forking satisfies weak uniqueness and weak extension (i.e.\ the conclusion of Fact \ref{ns-uq-ext} holds with ``split'' replaced by ``fork'').
  \end{enumerate}
\end{prop}
\begin{proof}
  Chase the definitions (not used).
\end{proof}

We also obtain a weak transitivity property:

\begin{prop}[Weak transitivity of splitting]\label{weak-trans-ns}
  Let $M_0 \lea M_1 \ltu M_1' \lea M_2$ all be in $K_\mu$. Let $p \in S (M_2)$. If $p \upharpoonright M_1'$ does not split over $M_0$ and $p$ does not split over $M_1$, then $p$ does not split over $M_0$.
\end{prop}
\begin{proof}
  By weak extension, find $q \in S (M_2)$ extending $p \upharpoonright M_1'$ and not splitting over $M_0$. By monotonicity, $q$ does not split over $M_1$. By weak uniqueness, $p = q$, as needed.
\end{proof}

We now turn to building a forking notion that will satisfy a version of uniqueness and extension (see Definition \ref{good-frame-def}) in $K_\mu$. The idea is simple enough: we want to say that a type does not fork over $M$ if there is a ``small'' substructure $M_0$ of $M$ over which the type does not \emph{split}. Fact \ref{ns-uq-ext} suggests that ``small'' should mean ``such that $M$ is a universal extension of $M_0$'', and this is exactly how we define it:

\begin{mydef}[$\mu$-forking]\label{nonforking-lambda}
  Let $M_0 \lea M \lea N$ be models in $K_\mu$. We say $p \in S (N)$ \emph{explicitly does not $\mu$-fork over $(M_0, M)$} if:

  \begin{enumerate}
    \item $M_0 \ltu M \lea N$.
    \item $p$ does not $\mu$-split over $M_0$.
  \end{enumerate}

  We say $p$ \emph{does not $\mu$-fork over $M$} if there exists $M_0$ so that $p$ explicitly does not $\mu$-fork over $(M_0, M)$.
\end{mydef}

The reader should note that the word ``forking'' is used in two different senses in this paper:

\begin{itemize}
  \item In the sense of an ``abstract notion'': this depends on a pre-$\mathcal{F}$-frame $\s$ and is called $\s$-forking in Definition \ref{s-forking}. This is defined for models of sizes in $\mathcal{F}$.
  \item In the concrete sense of Definition \ref{nonforking-lambda}. This is called $\mu$-forking and is only defined for types over models of size $\mu$. Later this will be extended to models of sizes at least $\mu$ and we will get a (concrete) notion called $(\ge \mu)$-forking (Definition \ref{gemu-forking}). Of course, the two notions will coincide over models of size $\mu$.
\end{itemize}

When we say that a type $p$ explicitly does not $\mu$-fork over $(M_0, M)$, we think of $M$ as the base, and $M_0$ as the \emph{explicit} witness to the $\mu$-nonforking. It would be nice if we could get rid of the witness entirely and get that $\mu$-nonforking satisfies extension and uniqueness, but uniqueness seems to depend on the particular witness.

Transitivity is also problematic: although we manage to get a weak version depending on the particular witnesses, we still do not know how to prove the witness-free version. This was stated as \cite[Exercise 12.9]{baldwinbook09} but Baldwin later realized \cite{baldwinbook09-errata} there was a mistake in his proof. 

If instead we define ``$p$ does not $\mu\text{-fork}^\ast$ over $M$'' to mean ``for all $M_0 \ltu M$ both in $K_\mu$ there exists $M_0'$ in $K_\mu$ with $M_0 \lea M_0' \ltu M$ and $p$ explicitly does not $\mu$-fork over $(M_0', M)$'' then extension and uniqueness (and thus transitivity) hold, but local character (assuming local character of splitting) is problematic. Thus it seems we have to carry along the witness in our definition of forking, and this makes the resulting independence notion quite weak (hence the name ``skeletal''). However, we will see in the next sections that (assuming some tameness and homogeneity) our skeletal $\mu$-frame transfers to a much better-behaved frame \emph{above $\mu$}. In particular, full uniqueness and transitivity will hold there.

\begin{lem}[Basic properties of $\mu$-forking]\label{basic-props-nf}
  Below, all models are in $K_\mu$.
  \begin{enumerate}
    \item \label{monot} Monotonicity: If $p \in S (N)$ explicitly does not $\mu$-fork over $(M_0, M)$, $M_0 \lea M_0' \lea M \lea M' \lea N' \lea N$ and $M_0' \ltu M'$, then $p \upharpoonright N'$ explicitly does not $\mu$-fork over $(M_0', M')$. In particular, if $p \in S (N)$ does not $\mu$-fork over $M$ and $M \lea M' \lea N' \lea N$, then $p \upharpoonright N'$ does not $\mu$-fork over $M'$.
    \item \label{extension} Extension: If $p \in S (N)$ explicitly does not $\mu$-fork over $(M_0, M)$ and $N' \gea N$, then there is $q \in S (N')$ extending $p$ that explicitly does not $\mu$-fork over $(M_0, M)$. If $p$ is nonalgebraic, then $q$ is nonalgebraic.
    \item \label{uniq} Uniqueness: If $p_\ell \in S (N)$ explicitly does not $\mu$-fork over $(M_0, M)$, $\ell = 1,2$, and $p_1 \upharpoonright M = p_2 \upharpoonright M$, then $p_1 = p_2$.
    \item \label{trans} Transitivity: Let $M_1 \lea M_2 \lea M_3$ and let $p \in S (M_3)$. If $p \upharpoonright M_2$ explicitly does not $\mu$-fork over $(M_0, M_1)$ and $p$ explicitly does not $\mu$-fork over $(M_0', M_2)$ for $M_0 \lea M_0'$, then $p$ explicitly does not $\mu$-fork over $(M_0, M_1)$.
    \item \label{nonalgebraicity} Nonalgebraicity: If $p \in S (N)$ does not $\mu$-fork over $M$ and $p \upharpoonright M$ is not algebraic, then $p$ is not algebraic. 
  \end{enumerate}
\end{lem}      
\begin{proof}
  Monotonicity follows directly from the definition (and Proposition \ref{ltl-basic-props}.(\ref{univ-trans})), extension and uniqueness are just restatements of Fact \ref{ns-uq-ext}, and transitivity is a restatement of Proposition \ref{weak-trans-ns}. For nonalgebraicity, assume $p \upharpoonright M$ is nonalgebraic. Then it has a nonalgebraic nonforking extension to $N$ by extension, and this extension must be $p$ by uniqueness, so the result follows.
\end{proof}

Assuming some local character for splitting, we obtain weak versions of the local character and continuity properties:

\begin{mydef}\label{lc-ns-def}
  Let $R$ be a binary relation on $K_\mu$, and let $\kappa$ be a regular cardinal. We say that $\mu$-splitting has \emph{$\kappa$-local character for $R$-increasing chains} if for any $R$-increasing $(M_i)_{i \le \delta}$ with $\text{cf} (\delta) \ge \kappa$, $M_\delta = \bigcup_{i < \delta} M_i$, and any $p \in S (M_\delta)$, there is $i < \delta$ so that $p$ does not split over $M_i$.
\end{mydef}
\begin{remark}\label{ns-local-stable}
  If $K_\mu$ is stable, then by \cite[Fact 4.6]{tamenessone} $\mu$-splitting has $\mu^+$-local character for $\le$-increasing chains.
\end{remark}

\begin{lem}\label{props-nf}
  Let $\ltg$ be an abstract universal ordering on $K_\mu$, and let $\kappa$ be a regular cardinal. Assume splitting has $\kappa$-local character for $\ltg$-increasing chains. Then:

  \begin{enumerate}
    \item $\kappa$-local character for $\ltg$-increasing chains: If $(M_i)_{i \le \delta}$ is a $\ltg$-increasing chain in $K_\mu$ with $\text{cf} (\delta) \ge \kappa$, $M_\delta = \bigcup_{i < \delta} M_i$ and $p \in S(M_\delta)$, then there exists $i < \delta$ so that $p$ explicitly does not $\mu$-fork over $(M_i, M_{i + 1})$.
    \item \label{continuity} $\kappa$-continuity for $\ltg$-increasing chains: If $(M_i)_{i \le \delta}$ is a $\ltg$-increasing chain in $K_\mu$ with $\text{cf}(\delta) \ge \kappa$, $M_\delta = \bigcup_{i < \delta} M_i$ and $p \in S (M_\delta)$ such that $p \upharpoonright M_i$ does not $\mu$-fork over $M_0$ for all $i < \delta$, then $p$ does not $\mu$-fork over $M_0$. Moreover, if in addition $p \upharpoonright M_i$ explicitly does not $\mu$-fork over $(M_0', M_0)$ for all $i < \delta$ (i.e.\ the witness is always the same), then $p$ explicitly does not $\mu$-fork over $(M_0', M_0)$.
    \item \label{props-nf-existence} Existence over $(\ge \kappa, \ltg)$-limits: If $M \in K_\mu$ is $(\delta, \ltg)$-limit for some $\delta$ with $\text{cf} (\delta) \ge \kappa$, then any $p \in S (M)$ does not $\mu$-fork over $M$. In fact, if $p_0, ..., p_{n - 1} \in S (M)$, $n < \omega$, then there exists $M_0 \ltu M$ such that $p_i$ explicitly does not $\mu$-fork over $(M_0, M)$ for all $i < n$.
  \end{enumerate}
\end{lem}
\begin{proof} \
  \begin{enumerate}
\item Follows from $\kappa$-local character of splitting for $\ltg$-increasing chains.
\item By $\kappa$-local character, there exists $i < \delta$ so that $p$ explicitly does not $\mu$-fork over $(M_i, M_{i + 1})$. By assumption, there exists $M_0' \ltu M_0$ so that $p \upharpoonright M_{i + 1}$ explicitly does not $\mu$-fork over $(M_0', M_0)$. Since $M_0' \lea M_i$, we can apply transitivity to obtain that $p$ explicitly does not $\mu$-fork over $(M_0', M_0)$. The proof of the moreover part is similar.
\item By local character and monotonicity.
  \end{enumerate}
\end{proof}

Thus if splitting has $\aleph_0$-local character for $\ltg$-increasing chains for some abstract universal ordering $\ltg$ and if all models in $K_\mu$ are $\ltg$-limit (e.g.\ if $K_\mu$ is categorical), then it seems we are very close to having a $\goodms{S}$ $\mu$-frame, but the witnesses must be carried along, which as observed above is rather annoying. Also, local character and continuity only hold for $\ltg$-chains.

In the next sections, we show that these problems disappear when we transfer our skeletal frame above $\mu$. Note that Shelah's construction of a good frame in \cite[Theorem II.3.7]{shelahaecbook} already takes advantage of that phenomenon. A similar idea is also exploited in the definition of a rooted minimal type in Grossberg and VanDieren's categoricity transfer from tameness \cite[Definition 2.6]{tamenessthree}.

\section{Going up without assuming tameness}\label{going-up-sec}

\begin{hypothesis} \
  \begin{enumerate}
    \item $K$ is an abstract elementary class. $\mu \ge \text{LS} (K)$ is a cardinal. $K_\mu \neq \emptyset$.
    \item $\ltg$ is an abstract universal ordering on $K_\mu$. In particular (by Remark \ref{rmk-good-ordering-stable}), $K_\mu$ has amalgamation, no maximal models, and is stable.
  \end{enumerate}  
\end{hypothesis}

In \cite[Section II.2]{shelahaecbook}, Shelah showed how to extend a good $\mu$-frame to all models in $K_{\ge \mu}$. The resulting object will in general not be a good $(\ge \mu)$-frame, but several of the properties are nevertheless preserved. In this section, we apply the same procedure on our skeletal $\mu$-frame (induced by $\mu$-forking defined in the previous section) and show Shelah's arguments still go through, assuming the base models are $\mu^+$-homogeneous. In the next section, we will assume tameness to prove more properties of $(\ge \mu)$-forking.

We define $(\ge \mu)$-forking from $\mu$-forking in exactly the same way Shelah extends a good $\mu$-frame to a $(\ge \mu)$-frame:

\begin{mydef}\label{gemu-forking}
  Assume $M, N \in K_{\ge \mu}$ and $p \in S^{\text{na}} (N)$. We say that $p$ \emph{does not $(\ge \mu)$-fork over $M$} if $M \lea N$ and there exists $M'$ in $K_\mu$ with $M' \lea M$ such that for all $N' \in K_\mu$ with $M' \lea N' \lea N$, $p \upharpoonright N'$ does not $\mu$-fork over $M'$.
\end{mydef}

For technical reasons, we also need to define explicit $(\ge \mu)$-forking over a model of size $\mu$:

\begin{mydef}[Explicit $(\ge \mu)$-forking in $K_{\ge \mu}$]
  Assume $N \in K_{\ge \mu}$, $M_0 \lea M$ are in $K_\mu$, and $p \in S^{\text{na}} (N)$. We say that $p$ \emph{explicitly does not $(\ge \mu)$-fork over $(M_0, M)$} if $p$ does not $\mu$-split over $M_0$ and $M_0 \ltu M \lea N$. Equivalently, for all $N' \in K_\mu$ with $M \lea N' \lea N$, we have $p \upharpoonright N'$ explicitly does not $\mu$-fork over $(M_0, M)$ (see Definition \ref{nonforking-lambda}).
\end{mydef}
\begin{remark}
  The following easy propositions follow from the definitions. We will use them without further comments in the rest of this paper.

  \begin{enumerate}
    \item The definitions of $(\ge \mu)$-forking and $\mu$-forking coincide over models of size $\mu$. That is, if $M_0, M, N \in K_\mu$ and $p \in S^{\text{na}} (N)$, then $p$ does not $\mu$-fork over $M$ if and only if $p$ does not $(\ge \mu)$-fork over $M$ and $p$ explicitly does not $(\ge \mu)$-fork over $(M_0, M)$ if and only if $p$ explicitly does not $\mu$-fork over $(M_0, M)$.
    \item For $M \lea N$ both in $K_{\ge \mu}$, $p \in S^{\text{na}} (N)$ does not $(\ge \mu)$-fork over $M$ if and only if there exists $M_0 \lea M$ in $K_\mu$ such that $p$ does not $(\ge \mu)$-fork over $M_0$.
    \item For $M \lea N$ with $M \in K_\mu$, $N \in K_{\ge \mu}$, $p \in S^{\text{na}} (N)$ does not $(\ge \mu)$-fork over $M$ if and only if for all $N' \lea N$ with $M \lea N'$, $p \rest N'$ does not $\mu$-fork over $M$.
  \end{enumerate}
\end{remark}

\begin{mydef}
  We define a nonforking relation $\nf$ on $K_{\ge \mu}$ by $\nfs{M}{a}{N}{\bigN}$ if and only if $M, N, \bigN \in K_{\ge \mu}$, $a \in \bigN$, and $\text{tp} (a / N; \bigN)$ does not $(\ge \mu)$-fork over $M$.
\end{mydef}

\begin{prop}\label{s-pre-frame}
  $\mathfrak{s}_0 := (K, \nf, S^{\text{na}})$ is a type-full pre-$[\mu, \infty)$-frame.
\end{prop}
\begin{proof}
  The properties to check follow directly from the definition of $(\ge \mu)$-nonforking. $\mathfrak{s}_0$ is type-full since we defined the basic types to be all the nonalgebraic types.
\end{proof}

In $K_\mu$ we had by definition that a type which does not $\mu$-fork over $M$ also explicitly does not $\mu$-fork over $(M_0, M)$ for some witness $M_0$. This is not necessarily the case for $(\ge \mu)$-nonforking: take for example $N \in K_{> \mu}$ and $M \in K_\mu$ and assume $p \in S (N)$ does not $(\ge \mu)$-fork over $M$. Then for all $N' \in K_\mu$ with $M \lea N' \lea N$, $p \upharpoonright N'$ does not $\mu$-fork over $M$, i.e.\ there is a witness $M_0'$ such that $p \upharpoonright N'$ explicitly does not $\mu$-fork over $(M_0', M)$, but there could be different witnesses $M_0'$ for different $N'$. 

The next lemma shows that this can be avoided if we have enough homogeneity. This is crucial to our proofs of transitivity, uniqueness, and extension.

\begin{lem}\label{unif-witness}
  Assume $M \lea N$ are both in $K_{\ge \mu^+}$ and $M$ is $\mu^+$-model-homogeneous. Assume $p \in S (N)$ does not $(\ge \mu)$-fork over $M$. Then there exists $M_0', M' \in K_\mu$ with $M_0' \lea M' \lea M$ such that $p$ explicitly does not $(\ge \mu)$-fork over $(M_0', M')$ (i.e. $(\ge \mu)$-nonforking over $M'$ is witnessed by the same $M_0'$ uniformly, see the discussion above).
\end{lem}
\begin{proof}
  By definition, there is $M_0'$ in $K_\mu$ with $M_0' \lea M$ such that $p$ does not $(\ge \mu)$-fork over $M_0'$. Since $M$ is $\mu^+$-model-homogeneous, one can pick $M' \gtu M_0'$ in $K_\mu$ with $M' \lea M$. By monotonicity (Lemma \ref{basic-props-nf}.(\ref{monot})), $p$ explicitly does not $(\ge \mu)$-fork over $(M_0', M')$.
\end{proof}

Using Lemma \ref{unif-witness}, we can give a simpler definition of $(\ge \mu)$-forking. This will not be used but shows that our forking is the same as that defined in \cite[Definition III.9.5.2]{shelahaecbook}.

\begin{prop}
  Assume $M \lea N$ are both in $K_{\ge \mu^+}$ and $M$ is $\mu^+$-model-homogeneous. Let $p \in S^{\text{na}} (N)$. Then $p$ does not $(\ge \mu)$-fork over $M$ if and only if there exists $M_0 \in K_\mu$ such that $M_0 \lea M$ and $p$ does not $\mu$-split over $M_0$.
\end{prop}
\begin{proof}
  If $p$ does not $(\ge \mu)$-fork over $M$, use Lemma \ref{unif-witness} to get $M_0', M' \in K_\mu$ with $M_0 \lea M' \lea M$ such that $p$ explicitly does not $(\ge \mu)$-fork over $(M_0, M')$. By definition, this means that $p$ does not $\mu$-split over $M_0$. Conversely, assume $M_0 \in K_\mu$ is such that $M_0 \lea M$ and $p$ does not $\mu$-split over $M_0$. Since $M$ is $\mu^+$-model-homogeneous, there exists $M' \in K_\mu$ such that $M_0 \ltu M' \lea M$. Thus $p$ explicitly does not $(\ge \mu)$-fork over $(M_0, M')$, so it does not $(\ge \mu)$-fork over $M$.
\end{proof}

\begin{lem}[Existence]\label{lambdap-existence}
  Let $M \in K_{\ge \mu^+}$ be $\mu^+$-model-homogeneous. Then $p \in S^{\text{na}} (M)$ if and only if $p$ does not $(\ge \mu)$-fork over $M$.
\end{lem}
\begin{proof}
  If $p$ does not fork over $M$, then $p$ is nonalgebraic by definition. Now assume $p$ is nonalgebraic. By \cite[Fact 4.6]{tamenessone}, there is $M_0' \in K_\mu$ with $M_0' \lea M$ such that $p$ does not $\mu$-split over $M_0$. Pick $M' \in K_\mu$ with $M' \gtu M_0'$ so that $M' \lea M$. This is possible by $\mu^+$-model-homogeneity. We have that $p$ explicitly does not $(\ge \mu)$-fork over $(M_0', M')$, so does not $(\ge \mu)$-fork over $M'$, as needed.
\end{proof}

\begin{lem}[Transitivity]\label{lambdap-trans}
  If $M_0 \lea M_1 \lea M_2$ are all in $K_{\ge \mu}$, $M_1$ is $\mu^+$-model-homogeneous, $p \in S^{\text{na}} (M_2)$ is such that $p \upharpoonright M_1$ does not $(\ge \mu)$-fork over $M_0$ and $p$ does not $(\ge \mu)$-fork over $M_1$, then $p$ does not $(\ge \mu)$-fork over $M_0$. 
\end{lem}
\begin{proof}
  Find $M_0' \in K_\mu$ with $M_0' \lea M_0$ such that $p \upharpoonright M_1$ does not $(\ge \mu)$-fork over $M_0'$. Using monotonicity and Lemma \ref{unif-witness}, we can also find $M_1', M_1'' \in K_\mu$ with $M_0' \lea M_1' \ltu M_1'' \lea M_1$ such that $p$ explicitly does not $(\ge \mu)$-fork over $(M_1', M_1'')$. By transitivity in $K_\mu$ (Lemma \ref{basic-props-nf}.(\ref{trans})), $p$ does not $(\ge \mu)$-fork over $M_0'$, and hence over $M_0$.
\end{proof}

\begin{lem}[Local character]\label{lambdap-local}
  Assume splitting has $\kappa$-local character for $\ltg$-increasing chains. If $\text{cf}(\delta) \ge \kappa$, $(M_i)_{i \le \delta}$ is an increasing chain in $K_{\ge \mu^+}$ with $M_\delta = \bigcup_{i < \delta} M_i$, $M_i$ is $\mu^+$-model-homogeneous for $i < \delta$, and $p \in S^{\text{na}} (M_\delta)$, then there is $i < \delta$ such that $p$ does not $(\ge \mu)$-fork over $M_i$.
\end{lem}
\begin{proof}
  Without loss of generality, $\delta$ is regular. If $\delta \ge \mu^+$, then $M_\delta$ is also $\mu^+$-model-homogeneous so one can pick $N^\ast \in K_\mu$ with $N^\ast \lea M_\delta$ witnessing existence (use Lemma \ref{lambdap-existence}) and find $i < \delta$ with $N^\ast \lea M_i$, so $p$ does not $(\ge \mu)$-fork over $M_i$ as needed. Now assume $\delta < \mu^+$. We imitate the proof of \cite[Claim II.2.11.5]{shelahaecbook}. Assume the conclusion fails. Build $(N_i)_{i \le \delta}$ $\ltg$-increasing continuous in $K_\mu$, $(N_i')_{i \le \delta}$ $\lea$-increasing continuous in $K_\mu$ such that for all $i < \delta$:

\begin{enumerate}
  \item $N_i \le M_i$.
  \item \label{lcproof-ni-mdelta} $N_i \le N_i' \le M_\delta$.
  \item \label{lcproof-forkcond} $p \upharpoonright N_{i + 1}'$ explicitly $\mu$-forks over $(N_i, N_{i + 1})$.
  \item \label{lcproof-ni-contain} $\bigcup_{j \le i} (N_j' \cap M_{i + 1}) \subseteq |N_{i + 1}|$.
\end{enumerate}

\paragraph{This is possible}
For $i = 0$, let $N_0 \in K_\mu$ be any model with $N_0 \le M_0$, and let $N_0' := N_0$. For $i$ limit, take unions. For the successor case, assume $i = j + 1$. Choose $N_i \le M_i$ satisfying (\ref{lcproof-ni-contain}) with $N_i \gtg N_j$ (possible since $M_i$ is $\mu^+$-model-homogeneous). By assumption, $p$ $(\ge \mu)$-forks over $M_i$, hence explicitly $(\ge \mu)$-forks over $(N_j, N_i)$, and so by definition of forking and monotonicity there exists $N_i' \in K_\mu$ with $M_\delta \gea N_i' \gea N_i$, $N_i' \gea N_j'$, and $p \upharpoonright N_i'$ explicitly $\mu$-forking over $(N_j, N_i)$. It is as required.

\paragraph{This is enough}

By local character in $K_\mu$, there is $i < \delta$ such that $p \upharpoonright N_\delta$ explicitly does not $\mu$-fork over $(N_i, N_{i + 1})$. By (\ref{lcproof-ni-mdelta}) and (\ref{lcproof-ni-contain}), $N_\delta' \lea N_\delta$. Thus $p \upharpoonright N_{i + 1}'$ explicitly does not $\mu$-fork over $(N_i, N_{i + 1})$, contradicting (\ref{lcproof-forkcond}).
  
\end{proof}

\begin{lem}[Continuity]\label{lambdap-continuity} Assume splitting has $\kappa$-local character for $\ltg$-increasing chains. If $\text{cf}(\delta) \ge \kappa$, $(M_i)_{i \le \delta}$ is an increasing chain in $K_{\ge \mu^+}$ with $M_\delta = \bigcup_{i < \delta} M_i$, $M_i$ $\mu^+$-model-homogeneous for $i < \delta$, and $p \in S (M_\delta)$ is so that $p \upharpoonright M_i$ does not $(\ge \mu)$-fork over $M_0$ for all $i < \delta$, then $p_\delta$ does not $(\ge \mu)$-fork over $M_0$.
\end{lem}
\begin{proof}
  In a type-full frame such as ours, this follows directly from $\kappa$-local character and transitivity, see \cite[Claim II.2.17.3]{shelahaecbook}.
\end{proof}
\begin{remark}
In the statements of local character and continuity, we assumed that $M_i$ was $\mu^+$-model-homogeneous for all $i < \delta$, but \emph{not} that their union $M_\delta$ was $\mu^+$-model-homogeneous.
\end{remark}

\section{A tame $\goodms{S}$ frame}\label{good-frame-without-symmetry}

Boney showed in \cite{ext-frame-jml} that given a good $\mu$-frame, tameness implies that Shelah's extension of the frame to $\ge \mu$ is actually a good $(\ge \mu)$-frame. In this section, we apply the ideas of his proof (assuming the base models are $\mu^+$-model-homogeneous) to our skeletal $\mu$-frame.

More precisely, we fix a cardinal $\lambda > \mu$, assume enough tameness, and build a $\goodms{S}$ $\lambda$-frame (i.e.\ we have all the properties of a good $\lambda$-frame except perhaps symmetry). We will prove symmetry in the next section. 

\begin{hypothesis}\label{hyp-nosym} \
  \begin{enumerate}
    \item \label{nempty-hyp} $K$ is an abstract elementary class. $\mu \ge \text{LS} (K)$ is a cardinal. $K_\mu \neq \emptyset$.
    \item \label{good-ordering-hyp} $\ltg$ is an abstract universal ordering on $K_\mu$. In particular (by Remark \ref{rmk-good-ordering-stable}), $K_\mu$ has amalgamation, no maximal models, and is stable.
    \item \label{ns-lc-hyp} $\kappa$ is the least regular cardinal such that splitting has $\kappa$-local character for $\ltg$-increasing chains in $K_\mu$.
    \item \label{amalg-lambdap} $\lambda > \mu$ is such that:
      \begin{enumerate}
        \item $K$ is $(\mu, \lambda)$-tame\footnote{Recall (Definition \ref{tameness-def}) that this means that the Galois types over models of size at most $\lambda$ are determined by their restrictions to submodels of size $\mu$.}.
        \item $K_{[\mu, \lambda]}$ has amalgamation.
        \item $K_{[\mu, \lambda)}$ has no maximal models.
      \end{enumerate}
  \end{enumerate}  
\end{hypothesis}
\begin{remark}
  $\kappa$ plays a similar role as the cardinal $\kappa(T)$ in the first-order context. By Remark \ref{ns-local-stable} and Hypothesis \ref{hyp-nosym}.(\ref{good-ordering-hyp}), $\kappa \le \mu^+$. In the end, we will be able to obtain a good frame only when $\kappa = \aleph_0$, but studying the general case leads to results on the stability spectrum.
\end{remark}

Note that uniqueness is actually \emph{equivalent} to $(\mu, \lambda)$-tameness by \cite[Theorem 3.2]{ext-frame-jml}. The easiest case is when $\lambda = \mu^+$. Then we know a model-homogeneous model exists in $K_\lambda$, and this simplifies some of the proofs.

\begin{lem}[Uniqueness]\label{lambdap-uniq}
  Let $M \lea N$ be models in $K_{[\mu, \lambda]}$. Let $p, q \in S (N)$. Assume $p \upharpoonright M = q \upharpoonright M$.

  \begin{enumerate}
    \item \label{first-uniq} If $M \in K_\mu$ and $p, q$ explicitly do not $(\ge \mu)$-fork over $(M_0, M)$ for some $M_0 \ltu M$, then $p = q$.
    \item \label{second-uniq} If $M \in K_{[\mu^+, \lambda]}$ is $\mu^+$-model-homogeneous and $p, q$ do not $(\ge \mu)$-fork over $M$, then $p = q$.
  \end{enumerate}
\end{lem}
\begin{proof}
  (\ref{first-uniq}) follows from uniqueness in $K_\mu$ (Lemma \ref{basic-props-nf}.(\ref{uniq})) and tameness. To see (\ref{second-uniq}), use monotonicity and Lemma \ref{unif-witness}, to find $M_0', M' \in K_\mu$ with $M_0' \ltu M' \lea M$ such that both $p$ and $q$ explicitly do not $(\ge \mu)$-fork over $(M_0', M')$. Now apply (\ref{first-uniq}).
\end{proof}

Interestingly, we already have enough machinery to obtain a stability transfer theorem. First recall:

\begin{fact}\label{lambdap-stability-2}
  $K_{\mu^+}$ is stable.
\end{fact}
\begin{proof}
  This could be done using the method of proof of Theorem \ref{lambdap-stability}, but this is also \cite[Theorem 1]{b-k-vd-spectrum}.
\end{proof}

Recall that $\kappa$ is the local character cardinal, see Hypothesis \ref{hyp-nosym}.(\ref{ns-lc-hyp}.

\begin{lem}\label{lambdap-stability-technical}
  Assume that $\lambda > \mu^+$, $\text{cf} (\lambda) \ge \kappa$, and there are unboundedly (in the same sense as in the statement of Lemma \ref{univ-embed}) many $\mu \le \lambda' < \lambda$ such that $K_{\lambda'}$ is stable. Then $K_\lambda$ is stable.  
\end{lem}
\begin{proof}
  Let $M \in K_\lambda$. By Lemma \ref{univ-embed}, $M$ can be embedded inside some $\widehat{M} \in K_\lambda$ which can be written as $\bigcup_{i < \lambda} M_i$, with $(M_i)_{i < \lambda}$ an increasing chain\footnote{Explicitly, we take $(N_i)_{i < \lambda}$ as given by Lemma \ref{univ-embed} for some $N_0 \lea M$ in $K_{\mu^+}$, and let $M_i := N_{i + 1}$. Note that the chain $(M_i)_{i < \lambda}$ will \emph{not} be continuous.} of $\mu^+$-model-homogeneous models in $K_{[\mu^+, \lambda)}$. From amalgamation, we know that Galois types can be extended, so $|S(M)| \le |S(\widehat{M})|$, and so we can assume without loss of generality that $M = \widehat{M}$. Let $(p_j)_{j < \lambda^+}$ be types in $S (M)$. By $\kappa$-local character, for each $j < \lambda^+$ there is $i_j < \lambda$ such that $p_j$ does not $(\ge \mu)$-fork over $M_{i_j}$. By the pigeonhole principle, we may assume $i_j = i_0$ for all $j < \lambda^+$. Taking $i_0$ bigger if necessary, we may assume that $K_{\|M_{i_0}\|}$ is stable. Thus $|S (M_{i_0})| \le \|M_{i_0}\| \le \lambda$, so by the pigeonhole principle again, we can assume that there is $q \in S (M_{i_0})$ such that $p_j \upharpoonright M_{i_0} = q$ for all $j < \lambda^+$. By uniqueness, $p_j = p_{j'}$ for each $j, j' < \lambda^+$, so the result follows.
\end{proof}

We can now prove that stability transfers up if the locality cardinal $\kappa$ of Hypothesis \ref{hyp-nosym}.(\ref{ns-lc-hyp}) is $\aleph_0$. Recall that $\lambda$ is the cardinal above $\mu$ fixed in Hypothesis \ref{hyp-nosym}.(\ref{amalg-lambdap}). Recall also that we already have stability in $\mu$ by Hypothesis \ref{hyp-nosym}.(\ref{good-ordering-hyp}).

\begin{thm}[The superstability theorem]\label{lambdap-stability}
  If $\kappa = \aleph_0$, then $K_{\lambda}$ is stable.
\end{thm}
\begin{proof}
  We work by induction on $\lambda$. If $\lambda = \mu^+$, this is Fact \ref{lambdap-stability-2} and if $\lambda > \mu^+$ this is given by Lemma \ref{lambdap-stability-technical} and the induction hypothesis.
\end{proof}

Assuming the generalized continuum hypothesis (GCH), we can also say something for arbitrary $\kappa$ (this will not be used):

\begin{thm}\label{main-thm-spectrum}
  Assume GCH. If $\lambda^{<\kappa} = \lambda$, then $K_\lambda$ is stable.
\end{thm}
\begin{proof}
  By induction on $\lambda$. If $\lambda = \mu^+$, this is Fact \ref{lambdap-stability-2}, so assume $\lambda > \mu^+$. By König's theorem, $\text{cf} (\lambda) \ge \kappa$. If $\lambda$ is successor, then $\lambda^\mu = \lambda$ by GCH, so by \cite[Corollary 6.4]{tamenessone}, $K$ is stable in $\lambda$. If $\lambda$ is limit there exists a sequence of successor cardinals $(\lambda_i)_{i < \text{cf} (\lambda)}$ increasing cofinal in $\lambda$ with $\lambda_0 \ge \mu^+$. Since without loss of generality $\kappa \le \mu^+$ (Remark \ref{ns-local-stable}), GCH implies that $\lambda_i^{<\kappa} = \lambda_i$, so by the induction hypothesis, $K$ is stable in $\lambda_i$ for all $i < \text{cf} (\lambda)$. Apply Lemma \ref{lambdap-stability-technical} to conclude.
\end{proof}

We now prove extension. This follows from compactness in the first-order case, but we make crucial use of the superstability hypothesis $\kappa = \aleph_0$ in the general case (recall from the hypotheses of this section that $\kappa$ is the local character cardinal for $\mu$-splitting).

\begin{lem}\label{ext-nf-types}
  Assume $\kappa = \aleph_0$. Let $\delta < \lambda^+$ be a limit ordinal. Assume $(M_i)_{i \le \delta}$ is an increasing continuous sequence in $K_{[\mu, \lambda)}$ with $M_0 \in K_\mu$. Let $(p_i)_{i < \delta}$ be an increasing continuous sequence of types with $p_i \in S (M_i)$ for all $i < \delta$, and $p_i$ explicitly does not $(\ge \mu)$-fork over $(M_0', M_0)$. Assume that one of the following holds:

    \begin{enumerate}
      \item \label{first-case-ext-nf-types} $(M_i)_{i < \delta}$ is $\ltg$-increasing in $K_\mu$.
      \item For all $i < \delta$, $M_{i + 1}$ is $\mu^+$-model-homogeneous.
    \end{enumerate}

    Then there exists a unique $p_\delta \in S (M_\delta)$ extending each $p_i$ and explicitly not $(\ge \mu)$-forking over $(M_0', M_0)$.
\end{lem}
\begin{proof}
  This is similar to the argument in \cite[Corollary 2.22]{tamenessthree}, but we give some details. We focus on (\ref{first-case-ext-nf-types}) (the proof of the other case is completely similar). Build by induction $(f_{i, j})_{i < j < \delta}$,  $(a_i)_{i < \delta}$, and increasing continuous $(N_i)_{i < \delta}$ such that for all $i < j < \delta$:

  \begin{enumerate}
    \item $M_i \lea N_i$, $a_i \in N_i$.
    \item $f_{i, j}: N_i \rightarrow N_j$.
    \item For $j < k < \delta$, $f_{j, k} \circ f_{i, j} = f_{i, k}$.
    \item $f_{i, j}$ fixes $M_i$.
    \item $f_{i, j} (a_i) = a_j$.
    \item $p_i = \text{tp} (a_i / M_i; N_i)$.
  \end{enumerate}

  \paragraph{This is enough} Let $(N_\delta, (f_{i, \delta}))_{i < \delta}$ be the direct limit of the system $(N_i, f_{i, j})_{i < j < \delta}$, and let $a_\delta := f_{0, \delta} (a_0)$, $p_\delta := \text{tp} (a_\delta / M_\delta; N_\delta)$. One easily checks that $p_\delta$ extends each $p_i$, $i < \delta$, and so using continuity for $\ltg$-increasing chains (Lemma \ref{props-nf}.(\ref{continuity})), explicitly does not $(\ge \mu)$-fork over $(M_0', M_0)$. Finally, $p_\delta$ is unique by Lemma \ref{lambdap-uniq}.

  \paragraph{This is possible} For $i = 0$, we take $a_0$ and $N_0$ so that $\text{tp} (a_0 / M_0 ; N_0) = p_0$. For $i$ limit, we let $(N_i, f_{i_0, i})_{i_0 < i}$ be the direct limit of the system $(N_{i_0}, f_{i_0, j_0})_{i_0 < j_0 < i}$, and let $a_i := f_{0, i} (a_0)$. By continuity for $\ltg$-increasing chains, $\text{tp} (a_i / M_i; N_i)$ explicitly does not $(\ge \mu)$-fork over $(M_0', M_0)$, and so by uniqueness, it must equal $p_i$. For $i = i_0 + 1$ successor, find $a_i$ and $N_i' \ge M_i$ such that $p_i = \text{tp} (a_i / M_i; N_i')$. Since $p_i \upharpoonright M_{i_0} = p_{i_0}$, we can use the definition of types to amalgamate $N_{i_0}$ and $N_i'$ over $M_{i_0}$: there exists $N_i \gea N_i'$ and $f_{i_0, i}: N_{i_0} \xrightarrow[M_{i_0}]{} N_i$ so that $f_{i_0, i} (a_{i_0}) = a_i$. Define $f_{i_0', i} := f_{i_0, i} \circ f_{i_0', i_0}$ for all $i_0' < i_0$.
\end{proof}

\begin{lem}[Extension]\label{lambdap-ext}
  Assume $\kappa = \aleph_0$. Let $M \lea N$ both be in $K_{[\mu^+, \lambda]}$ with $M$ and $N$ $\mu^+$-model-homogeneous, and let $p \in S^{\text{na}} (M)$. Then there is $q \in S (N)$ extending $p$ that does not fork over $M$.
\end{lem}
\begin{proof}
  We imitate the proof of \cite[Theorem 5.3]{ext-frame-jml}. By existence and Lemma \ref{unif-witness}, there exists $M_0', M_0 \in K_\mu$ with $M_0' \ltu M_0 \lea M$ and $p$ explicitly $(\ge \mu)$-nonforking over $(M_0', M_0)$. Work by induction on $\lambda$. If $N \in K_{<\lambda}$, use the induction hypothesis, so assume $N \in K_\lambda$. There are two cases: either $\lambda = \mu^+$ or $\lambda > \mu^+$. 

  Assume first $\lambda > \mu^+$. By transitivity and Lemma \ref{univ-embed}, we can assume without loss of generality that $N = \bigcup_{i < \lambda} N_i$, where $(N_i)_{i \le \lambda}$ is a $\ltu$-increasing continuous chain in $K_{[\mu^+, \lambda)}$, each $N_{i + 1}$ is $\mu^+$-model-homogeneous, and $N_0$ extends $M_0$. Now inductively build a $\le$-increasing continuous $(M_i)_{i \le \lambda}$ with $M_{\lambda} = M$ so that $M_0 \lea M_i \lea N_i$ for all $i < \lambda$ (we allow repetitions). Set $p_i := p \upharpoonright M_i$ and note that by monotonicity, $p_i$ explicitly does not $(\ge \mu)$-fork over $(M_0', M_0)$.

  We inductively build an increasing $(q_i)_{i \le \lambda}$ with $q_i \in S (N_i)$, $p_i \lea q_i$, and $q_i$ explicitly does not $(\ge \mu)$-fork over $(M_0', M_0)$. For $i = 0$, use extension in $K_{<\lambda}$ to find $q_0$ as needed. For $i = j + 1$, use extension to find a $(\ge \mu)$-nonforking extension $q_i \in S (N_i)$ of $q_j$ that explicitly does not $(\ge \mu)$-fork over $(M_0', M_0)$. By uniqueness, $q_i \gea p_i$. At limits, use Lemma \ref{ext-nf-types} and uniqueness. $q := q_\lambda$ is as desired.

  If $\lambda = \mu^+$, the construction is exactly the same except we use extension in $K_\lambda$ at successor steps and the first case of Lemma \ref{ext-nf-types} at limit steps. Note that since $N$ is $\mu^+$-model-homogeneous, $N = \bigcup_{i < \mu^+} N_i$, where $(N_i)_{i < \mu^+}$ is a $\ltg$-increasing continuous chain in $K_\mu$.
\end{proof}

\begin{mydef}
  Let $\mathfrak{s} := \mathfrak{s}_0 \upharpoonright \lambda$, where $\mathfrak{s}_0$ is the pre-frame from Proposition \ref{s-pre-frame}.
\end{mydef}

\begin{cor}\label{s-good-frame-nosym}
  Assume:

  \begin{enumerate}
    \item $\kappa = \aleph_0$.
    \item $K_\mu$ has joint embedding.
    \item $K_\lambda$ has no maximal models.
    \item All the models in $K_\lambda$ are $\mu^+$-model-homogeneous.
  \end{enumerate}
  
  Then $\mathfrak{s}$ is a type-full $\goodms{S}$ $\lambda$-frame.
\end{cor}
\begin{proof}
  It is easy to see $\mathfrak{s}$ is a type-full pre-$\lambda$-frame. $K_\lambda$ has amalgamation and no maximal models by hypothesis. It has joint embedding since $K_\mu$ has joint embedding and $K_{[\mu, \lambda]}$ has amalgamation (see Lemma \ref{jep-from-amalgamation}). Stability holds by Theorem \ref{lambdap-stability}. Density of basic types is always true in a type-full frame. For the other properties, see Lemmas \ref{lambdap-existence}, \ref{lambdap-trans}, \ref{lambdap-local}, \ref{lambdap-continuity}, \ref{lambdap-uniq}, and \ref{lambdap-ext} (note that the original statement of extension in Definition \ref{good-frame-def} follows from Lemma \ref{lambdap-ext} and transitivity).
\end{proof}

\begin{lem}\label{categ-conseq}
  Assume $K$ is categorical in $\lambda$ and $\kappa = \aleph_0$. Then:

  \begin{enumerate}
    \item \label{categ-conseq-1} $K_{[\mu, \lambda]}$ has joint embedding and $K_\lambda$ (and hence $K_{[\mu, \lambda]}$) has no maximal models.
    \item \label{categ-conseq-2} All the models in $K_\lambda$ are $\mu^+$-model-homogeneous.
  \end{enumerate}
\end{lem}
\begin{proof}
  To see (\ref{categ-conseq-2}), assume first that $K_\lambda$ has no maximal models. Use stability to build $(M_i)_{i \le \mu^+}$ $\ltu$-increasing continuous with $M_i \in K_{\lambda}$ for all $i < \mu^+$. Then $M_{\mu^+}$ is $\mu^+$-model-homogeneous. If $K_\lambda$ has a maximal model, then it is easy to check directly that the maximal model is $\mu^+$-model-homogeneous.

    For (\ref{categ-conseq-1}), $K_\lambda$ has joint embedding by categoricity. Now since $K_{[\mu, \lambda)}$ has no maximal models, any $M \in K_{[\mu, \lambda)}$ embeds into an element of $K_\lambda$, so joint embedding for $K_{[\mu, \lambda]}$ follows . To see $K_\lambda$ has no maximal model, let $N \in K_\lambda$ be given. First assume $\lambda = \mu^+$. Build a $\ltg$-increasing continuous chain $(M_i)_{i \le \mu^+}$, and $a \in N$ such that for all $i < \mu^+$:

  \begin{enumerate}
    \item $M_i \in K_\mu$, $M_i \lea N$.
    \item $a \notin M_0$.
    \item $\text{tp} (a / M_i; N)$ does not $\mu$-fork over $M_0$.
  \end{enumerate}

  \paragraph{This is enough} $M_{\mu^+} \in K_{\lambda^+}$. Moreover by Lemma \ref{basic-props-nf}.(\ref{nonalgebraicity}), $a \notin M_i$ for all $i < \mu^+$, so $a \notin M_{\mu^+}$. Thus $M_{\mu^+} \lta N$. By categoricity, the result follows.

  \paragraph{This is possible} Pick a $\ltg$-limit $M_0 \in K_\mu$ with $M_0 \lea N$ (this is possible by model-homogeneity of $N$), and pick any $a \in N \backslash M_0$. At limits, take unions and use continuity (Lemma \ref{props-nf}.(\ref{continuity})) to see the requirements are maintained. For a successor $i = j + 1$, use extension and some renaming. In details, pick an arbitrary $M_i' \gtg M_j$ with $M_i' \lea N$ (possible by model-homogeneity). By extension (Lemma \ref{basic-props-nf}.(\ref{extension})), there is $q \in S (M_i')$ that does not $\mu$-fork over $M_0$ and extends $p_j := \text{tp} (a / M_j; N)$.  Since $N$ is saturated, there is $a' \in N$ realizing $q$. Pick $N \gea N_i \gea M_i'$ containing $a'$ and $a$. By assumption, $\text{tp} (a' / M_j; N_i) = p_j = \text{tp} (a / M_j; N_i)$. Thus there is $N_i' \gea N_i$ and $f: N_i \xrightarrow[M_j]{} N_i'$ such that $f (a') = a$ and without loss of generality $N_i' \lea N$. Let $M_i := f[M_i']$ and use invariance to see it is as desired.

  If $\lambda > \mu^+$, the proof is completely similar: if there is $N_1 \gta N$, we are done, so assume not. Then amalgamation implies $N$ must be model-homogeneous. Build a $\ltu$-increasing continuous $(M_i)_{i \le \lambda}$ and $a \in N$ such that for all $i < \lambda$:

  \begin{enumerate}
    \item $M_i \in K_{[\mu^+, \lambda)}$, $M_i \lea N$.
    \item $M_{i + 1}$ is $\mu^+$-model-homogeneous.
    \item $\text{tp} (a / M_i; N)$ does not $(\ge \mu)$-fork over $M_0$.
  \end{enumerate}

  As before, this is possible and the result follows.
\end{proof}

\begin{cor}
  If $K$ is categorical in $\lambda$ and $\kappa = \aleph_0$, then $\mathfrak{s}$ is a type-full $\goodms{S}$ $\lambda$-frame.
\end{cor}
\begin{proof}
  Lemma \ref{categ-conseq} tells us all the hypotheses of Corollary \ref{s-good-frame-nosym} are satisfied.
\end{proof}

Note that categoricity in $\lambda$ is not the only hypothesis giving that all models in $K_\lambda$ are $\mu^+$-model-homogeneous. For example:

\begin{fact}[Theorem 5.4 in \cite{bg-v7}]
  Assume $K$ has amalgamation, is categorical in a cardinal $\theta$ so that $K_\theta$ has a $\mu^+$-model-homogeneous model (this holds if e.g. $\theta^{\mu} = \theta$). Then every member of $K_{\ge \chi}$ is $\mu^+$-model-homogeneous, where $\chi := \min (\theta, \sup_{\gamma < \mu} \beth_{(2^{\gamma})^+})$.
\end{fact}

\section{Getting symmetry}\label{getting-sym}

From Corollary \ref{s-good-frame-nosym}, we obtain from reasonable assumptions a forking notion that satisfies all the properties of a good $\lambda$-frame except perhaps symmetry. Note that assuming more tameness, the frame can also be extended (see Fact \ref{frame-ext}) to models of size above $\lambda$:

\begin{fact}\label{ext-frame}
  Let $\s = (K, \nf, S^{\text{bs}})$ be a $\goodms{S}$ $\lambda$-frame. Let $\theta > \lambda$ and let $\mathcal{F} := [\lambda, \theta)$. Assume $K_{\mathcal{F}}$ has amalgamation and no maximal models, and $K$ is $(\lambda, <\theta)$-tame. Then $\mathfrak{s}$ can be extended to a $\goodms{S}$ $\mathcal{F}$-frame. If $\mathfrak{s}$ is type-full, then the extended frame will also be type-full. 
\end{fact}
\begin{proof}
  Apply \cite[Theorem 1.1]{ext-frame-jml}: its proof only uses the tameness for $2$-types hypothesis to obtain symmetry. Note that if (as there) we start with a good $\lambda$-frame, then no maximal models follows. Here we do not have symmetry, so we assume it as an additional hypothesis. The proof of Lemma \ref{lambdap-local} gives us that the extended frame is type-full if $\mathfrak{s}$ is.
\end{proof}

 We have justified:

\begin{hypothesis}\label{good-frame-hyp}
  $\mathfrak{s} = (K, \nf, S^{\text{bs}})$ is a $\goodm$ $\mathcal{F}$-frame, where $\mathcal{F}$ is an interval of cardinals of the form $[\lambda, \theta)$ for $\lambda$ a cardinal and $\theta > \lambda$ either a cardinal or $\infty$.
\end{hypothesis}

In this section, we will prove that $\mathfrak{s}$ also satisfies symmetry if $\theta$ is big-enough. Note that we do not need to assume tameness since enough tameness for what we want follows from the uniqueness and local character properties of $\s$-forking, see \cite[Theorem 3.2]{ext-frame-jml}.

Note that (see the definition of $\goodm$ in \ref{good-frame-def}) we do \emph{not} assume $\mathfrak{s}$ satisfies bs-stability. It will hold in the setup of the previous sections, but the arguments of this section work just as well without it. Note in passing that bs-stability and stability are equivalent:

\begin{fact}[\cite{shelahaecbook}, Claim II.4.2.1]\label{frame-stability}
  For any $\lambda' \in \mathcal{F}$, $\mathfrak{s} \upharpoonright \lambda'$ satisfies bs-stability if and only if $K$ is stable in $\lambda'$.
\end{fact}

Moreover, eventual stability will follow from the structural properties of forking:

\begin{prop}\label{eventual-stability} \
  \begin{enumerate}
  \item If $2^\lambda \in \mathcal{F}$, then $K$ is stable in $2^\lambda$.
  \item Assume $\chi_0 \in \mathcal{F}$ and $K$ is stable in $\chi_0$. Then $K$ is stable in every $\chi \ge \chi_0$ with $\chi \in \mathcal{F}$.
  \end{enumerate}

  In particular, if $\chi$ is a cardinal with $2^\lambda \le \chi < \theta$, then $K$ is stable in $\chi$.
\end{prop}
\begin{proof} \
  \begin{enumerate}
  \item Let $\chi := 2^\lambda$. By Fact \ref{frame-stability}, it is enough to show that $\mathfrak{s} \upharpoonright \chi$ satisfies bs-stability. Let $M \in K_{\chi}$, and let $(p_i)_{i < \chi^+}$ be elements of $S^{\text{bs}} (M)$. Let $(M_i)_{i < \chi}$ be a resolution of $M$. For each $i < \chi^+$, local character implies there exists $j_i < \chi$ such that $p_i$ does not $\s$-fork over $M_{j_i}$. By the pigeonhole principle, we can assume without loss of generality that $j_i = j_0$ for all $i < \chi^+$. By Proposition \ref{kappabar} and transitivity, there exists $M' \in K_\lambda$ such that $M' \lea M_{j_0}$ and $p_i$ does not $\s$-fork over $M'$ for all $i < \chi^+$. We know that $|S (M')| \le 2^\lambda = \chi$, so by the pigeonhole principle again, we can assume that there is $q \in S (M')$ such that $p_i \upharpoonright M' = q$ for all $i < \chi^+$. By uniqueness, $p_i = p_{i'}$ for all $i, i' < \chi^+$, and the result follows.
  \item By the proof of stability in Fact \ref{ext-frame}.
  \end{enumerate}
\end{proof}

We would like to give conditions under which $\mathfrak{s}$ has symmetry. A useful fact\footnote{This is not crucial to our argument, but enables us to obtain an explicit upper bound on the amount of tameness needed.} is that it is enough to look at $\mathfrak{s} \upharpoonright \lambda$:

\begin{fact}[Theorem 6.8 in \cite{tame-frames-revisited-v4}]\label{symm-lambda}
  $\mathfrak{s}$ has symmetry if and only if $\mathfrak{s} \upharpoonright \lambda$ has symmetry.
\end{fact}

Since we are not assuming anything about how $\mathfrak{s}$ is defined, we will work by contradiction: We will show that if $\theta$ is big enough and symmetry fails, then we get the order property, a nonstructure property which implies unstability. This is how the symmetry property of forking was originally proven in the first-order context, see \cite[Theorem III.4.13]{shelahfobook}. The same approach was later used in a non-elementary setup in \cite[Theorem 5.1]{sh48}, and generalized in \cite[Theorem 5.14]{bgkv-v2}. We will rely on the proof of the latter.

The definitions and fact below do not need Hypothesis \ref{good-frame-hyp}.

\begin{mydef}
  Let $\alpha$, $\chi$ and $\gamma$ be cardinals. A model $N$ has the \emph{$(\alpha, \chi)$-order property of length $\gamma$} if there exists $M \in K_{\le \chi}$ with $M \lea N$ (we also allow $M$ to be empty) and $(\bar{a}_i)_{i < \gamma}$, $\bar{a}_i \in \fct{\alpha}{N}$ so that for any $i_0 < i_1 < \gamma$ and $j_0 < j_1 < \gamma$, $\text{tp} (\bar{a}_{i_0} \bar{a}_{i_1} / M; N) \neq \text{tp} (\bar{a}_{j_1} \bar{a}_{j_0} / M; N)$. If $\chi = 0$, we omit it. 

  $K$ has the \emph{$(\alpha, \chi)$-order property of length $\gamma$} if some $N \in K$ has it. $K$ has the \emph{$(\alpha, \chi)$-order property} if it has the $(\alpha, \chi)$-order property for all lengths (we sometimes also say $K$ has the $(\alpha, \chi)$-order property of length $\infty$). $K$ has the \emph{order property} if it has the $\alpha$-order property for some $\alpha$.
\end{mydef}

\begin{remark}\label{op-rmk}
If $N$ has the $(\alpha, \chi)$-order property of length $\gamma$, then it has the $(\alpha + \chi)$-order property of length $\gamma$. 
\end{remark}

\begin{mydef}
  Given a cardinal $\chi$, define $\hanf{\chi} := \beth_{(2^{\chi})^+}$. 
\end{mydef} 

\begin{fact}\label{op-basic-facts} \
  \begin{enumerate}
    \item \label{op-length} If $K$ has the $(\alpha, \chi)$-order property of length $\hanf{\alpha + \chi + \text{LS} (K)}$, then $K$ has the $(\alpha, \chi)$-order property.
    \item\label{op-unstable} If $K$ has the $(\alpha, \chi)$-order property, then it is $\alpha$-unstable in $\chi'$ for all $\chi' \ge \chi$.
  \end{enumerate}
\end{fact}
\begin{proof}
  The statements essentially appear in \cite[Claim 4.5.3, Claim 4.7.2]{sh394}. The proof of (\ref{op-length}) is an application of Morley's method together with Shelah's presentation theorem, and a proof of a statement similar to (\ref{op-unstable}) is sketched in \cite[Fact 5.13]{bgkv-v2}.
\end{proof}

\begin{fact}\label{nosym-op}
  If $\mathfrak{s}$ does not have symmetry, then $K$ has the $(2, \lambda)$-order property of length $\theta$.
\end{fact}
\begin{proof}
  By Fact \ref{symm-lambda}, $\mathfrak{s} \upharpoonright \lambda$ does not have symmetry. The result now follows by exactly the same proof as \cite[Theorem 5.14]{bgkv-v2}.
\end{proof}
\begin{cor}\label{lambdap-sym}
  If $\theta \ge \hanf{\lambda}$, then $\mathfrak{s}$ has symmetry.
\end{cor}
\begin{proof}
  If $\mathfrak{s}$ does not have symmetry, then by Fact \ref{nosym-op} and Fact \ref{op-basic-facts}.(\ref{op-length}), $K$ has the $(2, \lambda)$-order property and hence by Fact \ref{op-basic-facts}.(\ref{op-unstable}) is $2$-unstable in $2^\lambda$. By Theorem \ref{stab-longtypes}, $K$ is unstable in $2^\lambda$, contradicting Proposition \ref{eventual-stability} (note that $2^\lambda < h (\lambda) \le \theta$).
\end{proof}

Thus it seems quite a big gap between $\lambda$ and $\theta$ is needed. On the other hand the proof of Fact \ref{ext-frame} tells us that with enough tameness we can make $\mathcal{F}$ bigger:

\begin{fact}\label{frame-ext}
  Let $\theta' \ge \theta$ and let $\mathcal{F}' := [\lambda, \theta')$. Assume $K_{\mathcal{F}'}$ has amalgamation and no maximal models, and $K$ is $(\lambda, \theta')$-tame. Then $\mathfrak{s}$ can be extended to a $\goodm$ $[\lambda, \theta')$-frame. If $\mathfrak{s}$ has bs-stability, the extended frame will also have bs-stability. If $\mathfrak{s}$ is type-full, then the extended frame will also be type-full. 
\end{fact}
\begin{proof}
  By Remark \ref{frame-up}, $\mathfrak{s}$ is determined by $\mathfrak{s} \upharpoonright \lambda$. Now apply Fact \ref{ext-frame}. 
\end{proof}

\begin{remark}
  We could replace $(\lambda, \theta')$-tameness by $(\lambda', \theta')$-tameness in the above, where $\lambda' \in \mathcal{F}$. This turns out to be equivalent (at least if we consider tameness for basic types) since the uniqueness property of $\mathfrak{s}$ gives us $(\lambda, \lambda')$-tameness for basic types.
\end{remark}

\begin{cor}\label{tame-sym}
  Let $\mathcal{F}' := [\lambda, \hanf{\lambda})$. Assume $K_{\mathcal{F}'}$ has amalgamation and no maximal models, and $K$ is $(\lambda, <\hanf{\lambda})$-tame. Then $\mathfrak{s}$ has symmetry.
\end{cor}
\begin{proof}
  Using Fact \ref{frame-ext}, we can extend $\s$ to assume without loss of generality that $\theta \ge \hanf{\lambda}$. Now use Corollary \ref{lambdap-sym}.
\end{proof}

\section{The main theorems}\label{main-thm}

We finally have our promised good frame:

\begin{thm}\label{main-thm-technical}
  Assume:

  \begin{enumerate}
    \item $K$ is an abstract elementary class. $\mu \ge \text{LS} (K)$ is a cardinal.
    \item $K_\mu \neq \emptyset$ has joint embedding.
    \item $\ltg$ is an abstract universal ordering on $K_\mu$. In particular (by Remark \ref{rmk-good-ordering-stable}), $K_\mu$ has amalgamation, no maximal models, and is stable.
    \item Splitting has $\aleph_0$-local character for $\ltg$-increasing chains in $K_\mu$.
    \item $\lambda > \mu$ is such that:
      \begin{enumerate}
        \item $K$ is $(\mu, < \hanf{\lambda})$-tame.
        \item $K_{[\mu, \hanf{\lambda})}$ has amalgamation and no maximal models.
        \item All the models in $K_\lambda$ are $\mu^+$-model-homogeneous.
      \end{enumerate}
  \end{enumerate}  

  Then $K$ has a type-full good $[\lambda, \hanf{\lambda})$-frame.
\end{thm}
\begin{proof}
  Corollary \ref{s-good-frame-nosym}, gives us a $\goodms{S}$ $\lambda$-frame $\mathfrak{s}$. By Corollary \ref{tame-sym}, $\mathfrak{s}$ also has symmetry.
\end{proof}

We can use categoricity to derive some of the hypotheses above. We will use:

\begin{fact}\label{categ-facts}
  Assume $K$ has amalgamation and no maximal models. Assume $K$ is categorical in $\lambda$. Then:

  \begin{enumerate}
    \item \label{categ-stable} $K$ is stable in all $\text{LS} (K) \le \mu < \lambda$.
    \item \label{categ-ns} For any $\text{LS} (K) \le \mu < \text{cf} (\lambda)$ and any limit $\delta < \mu^+$, $\mu$-splitting has $\aleph_0$-local character for $\ltg$-chains, where $\ltg := \ltl{\mu}{\delta}$.
    \item \label{categ-down} Let $\hanfn{2} := \hanf{\hanf{\text{LS} (K)}}$. Assume $\lambda$ is a successor cardinal and $\lambda > \lambda_0 \ge h_2$. Then $K$ is $(\hanfn{2}, \lambda_0)$-tame and categorical in $\lambda_0$. In addition, the model of size $\lambda_0$ is saturated.
  \end{enumerate}
\end{fact}
\begin{proof}
  (\ref{categ-stable}) is \cite[Claim 1.7]{sh394}. (\ref{categ-ns}) is \cite[Lemma 6.3]{sh394}, and (\ref{categ-down}) were originally stated (with a lower Hanf number) in \cite[Main Claim II.2.3]{sh394} and \cite[Theorem II.2.7]{sh394}. A full proof (with discussion on whether it is possible to lower the $\hanfn{2}$ bound) can be found in \cite[Chapter 14]{baldwinbook09}.
\end{proof}

\begin{thm}\label{main-thm-technical-categ}
  Let $K$ be an abstract elementary class and let $\lambda$ be a cardinal such that $\text{cf} (\lambda) > \mu \ge \text{LS} (K)$. Let $\mathcal{F} := [\lambda, \hanf{\lambda})$, $\mathcal{F}' := [\mu, \hanf{\lambda})$. Assume:

  \begin{enumerate}
  \item $K_{\mathcal{F}'}$ has amalgamation and no maximal models.
  \item $K_\lambda$ is categorical.
  \item \label{tameness-hyp} $K$ is $(\mu, <\hanf{\lambda})$-tame.
  \end{enumerate}

  Then $K$ has a type-full good $\mathcal{F}$-frame.
\end{thm}
\begin{proof}
  First, $K_\lambda \neq \emptyset$ by categoricity. By Lemma \ref{categ-conseq}, $K_{\mathcal{F}'}$ has joint embedding and all models in $K_\lambda$ are $\mu^+$-model-homogeneous. By Fact \ref{categ-facts}, $\mu$-splitting has $\aleph_0$-local character for $\ltg$-chains, where $\ltg := \ltl{\lambda}{\omega}$. This shows all the hypotheses of Theorem \ref{main-thm-technical} are satisfied.
\end{proof}

Assuming categoricity in a high-enough successor, we obtain the tameness assumption:

\begin{thm}\label{main-thm-categ}
  Let $K$ be an abstract elementary class. Let $\mu := \hanfn{2} := \hanf{\hanf{\text{LS} (K)}}$. Let $\lambda := \mu^+$. Assume $K$ has amalgamation, joint embedding, and is categorical in some successor $\theta \ge \hanf{\lambda}$.

  Let $\mathcal{F} := [\lambda, \theta)$. Then there is a type-full good $\mathcal{F}$-frame with underlying AEC $K$.
\end{thm}
\begin{proof}
  Since $\theta \ge \hanf{\text{LS} (K)}$, $K$ has arbitrarily large models and so using joint embedding $K$ has no maximal models. By Fact \ref{categ-facts}, $K$ is categorical in $\lambda$ and $K$ is $(\mu, <\hanf{\lambda})$-tame. Apply Theorem \ref{main-thm-technical-categ}.
\end{proof}

Notice that one also obtains that categoricity (at a cardinal of high-enough cofinality) and tameness implies stability everywhere. This improves on \cite[Corollary 4.7]{b-k-vd-spectrum}:

\begin{thm}\label{main-thm-superstable}
  Let $K$ be an abstract elementary class with amalgamation and no maximal models. Assume $K$ is categorical in some $\lambda$ such that $\text{cf} (\lambda) > \mu \ge \text{LS} (K)$ and $K$ is $(\mu, \mu')$-tame. Then $K$ is stable in all $\theta \in [\text{LS} (K), \mu']$. In particular, if $\mu' = \infty$, then $K$ is stable everywhere.
\end{thm}
\begin{proof}
  By Fact \ref{categ-facts}, $\mu$-splitting has $\aleph_0$-local character for $\ltg$-chains, where $\ltg := \ltl{\mu}{\omega}$ and $K$ is stable everywhere below and at $\mu$. Apply Theorem \ref{lambdap-stability} to see $K$ is stable everywhere in $(\mu, \mu']$.
\end{proof}

This result is much more local than the other results of this section. For example, we do not need to assume that $\mu' \ge \hanf{\mu}$. Moreover, as Theorem \ref{lambdap-stability} shows, the categoricity hypothesis can be replaced by $\mu$-splitting having $\aleph_0$-local character for $\ltg$-chains, for some abstract universal ordering $\ltg$ on $K_\mu$.

Assuming the generalized continuum hypothesis (GCH), we obtain a more general stability spectrum theorem:

\begin{thm}\label{main-thm-spectrum-2}
  Assume GCH. Let $K$ be an abstract elementary class with amalgamation and no maximal models. Assume $K$ is $\mu$-tame for $\mu \ge \text{LS} (K)$, $\ltg$ is an abstract universal ordering on $K_\mu$, and $\mu$-splitting has $\kappa$-local character for $\ltg$-increasing chains. Then $K$ is stable in all $\lambda \ge \mu$ with $\lambda = \lambda^{<\kappa}$.
\end{thm}
\begin{proof}
  $K$ is stable in $\mu$ since we have an abstract universal ordering on $K_\mu$. If $\lambda > \mu$, the result follows from Theorem \ref{main-thm-spectrum}.
\end{proof}
\begin{remark}
  If $K$ is the class of models of a complete first-order theory, the conditions for stability given by Corollary \ref{main-thm-spectrum-2} are very close\footnote{The least regular cardinal $\kappa$ such that splitting has $\kappa$-local character will be at most the \emph{successor} of $\kappa (T)$.} to optimal (see \cite[Corollary III.3.8]{shelahfobook}).
\end{remark}
\begin{remark}
  Let $K$ be an abstract elementary class with amalgamation and no maximal models. Assume $K$ is $\chi$-tame and stable in some $\mu \ge \hanf{\chi}$. Then \cite[Theorem 4.13]{tamenessone} shows that for some $\kappa < \hanf{\chi}$, $\mu$-splitting has $\kappa$-local character. Thus we have:
\end{remark}
\begin{cor}\label{main-thm-spectrum-3}
  Assume GCH. Let $K$ be an abstract elementary class with amalgamation and no maximal models. Assume $K$ is $\chi$-tame and stable in some $\mu \ge \text{LS} (K)$. Then there is $\kappa < \hanf{\chi}$ such that $K$ is stable in all $\lambda \ge \mu$ with $\lambda^{<\kappa} = \lambda$.
\end{cor}
\begin{proof}
  If $\mu < \hanf{\chi}$, then by \cite[Corollary 6.4]{tamenessone} one can take $\kappa := \mu^+$, so assume $\mu \ge \hanf{\chi}$. By the previous remark, there is $\kappa < \hanf{\chi}$ such that $\mu$-splitting has $\kappa$-local character. The result now follows from Theorem \ref{main-thm-spectrum-2}.
\end{proof}
\begin{remark}
In \cite{sv-infinitary-stability-v3}, we use different methods to prove Corollary \ref{main-thm-spectrum-3} in ZFC. We do not know whether Corollary \ref{main-thm-spectrum-2} also holds in ZFC (although it is clear from the proof that much less than GCH is needed).
\end{remark}

We can also apply our good frame to the question of uniqueness of limit models:

\begin{thm}[Uniqueness of limit models]\label{uniq-limit}
  Assume the hypotheses of Theorem \ref{main-thm-technical-categ} hold. Then $K$ has a unique limit model in any $\mu' \in \mathcal{F}$. In fact, if $M_0 \in K_{\mu'}$ and $M_\ell$ is $(\mu', \delta_l)$-limit over $M_0$ for $\ell = 1,2$ and $\delta_l$ a limit ordinal, then $M_1 \cong_{M_0} M_2$.

  In particular, if $K$ has amalgamation and no maximal models, is categorical in $\lambda$ and is $\mu$-tame for some $\mu < \text{cf} (\lambda)$, then $K$ has a unique limit model in any $\mu' \ge \lambda$.

\end{thm}
\begin{proof}
  By Theorem \ref{main-thm-technical-categ}, $K$ has a good $\mathcal{F}$-frame $\mathfrak{s}$. In particular, $K$ is stable in $\mu'$, so one can iterate Fact \ref{gtu-existence} to build a $(\mu', \delta)$-limit model for any desired $\delta < \left(\mu'\right)^+$. To see uniqueness, apply \cite[Lemma II.4.8]{shelahaecbook} (see \cite[Theorem 9.2]{ext-frame-jml} for a detailed proof of that result).
\end{proof}

We see this theorem as an encouraging approximation to generalizing the upward categoricity transfer result of \cite{tamenessthree} (which assumes categoricity in a successor cardinal) to categoricity in a limit cardinal.

\begin{remark}
  Uniqueness of limit models of cardinality $\mu$ was asserted to follow from categoricity in some $\lambda^+ > \mu$ already in \cite{shvi635}. However, an error was found by VanDieren in 1999. VanDieren \cite{vandierennomax,nomaxerrata} proves uniqueness with the additional assumption that unions of amalgamation bases are amalgamation bases (but does not use tameness). It is still open whether uniqueness of limit models follows from categoricity only. In \cite{gvv-v3}, it is shown that uniqueness of limit models follows from a superstability-like assumption akin to $\aleph_0$-local character of $\mu$-splitting, amalgamation, and a unidimensionality assumption (the authors initially claimed to prove the result without unidimensionality but the claim was later retracted). 
\end{remark}
\begin{remark}
  A variation on Theorem \ref{uniq-limit} is \cite[Corollary 6.10]{bg-v7}, which uses stronger locality assumptions but manages to obtain uniqueness of limit models below the categoricity cardinal without any cofinality restriction.
\end{remark}

\section{Conclusion and further work}\label{conclusion}

Assuming amalgamation, joint embedding, no maximal models, and tameness, we have given superstability-like conditions under which an abstract elementary class has a type-full good frame $\mathfrak{s}$, i.e.\ a forking-like notion for 1-types. These arguments would work just as well to get a notion of independence for all $n$-types, with $n < \omega$. The proof of extension breaks down, however, for types of infinite length (difficulties in obtaining the extension property in the absence of compactness is one of the reasons\footnote{Another reason was Shelah's example (see \cite[Section 4]{hyttinen-lessmann}) of an $\aleph_0$-stable non-simple diagram, but we have shown that we do not get into trouble as long as we restrict the base of our types to be sufficiently saturated models.} it was assumed as an axiom in \cite{bg-v7}).

Shelah's approach around this in \cite[Chapter II]{shelahaecbook} is to show that if the frame is \emph{weakly successful} (a uniqueness condition for certain kinds of amalgamations), then it has a notion of forking for types of models. In \cite[Chapter III]{shelahaecbook}, Shelah has several hundreds of pages of approximations on when weak successfulness can be transferred across cardinals (many of his difficulties come from the fact he is not assuming amalgamation or no maximal models), but even assuming $\mathfrak{s} \upharpoonright \lambda$ is weakly successful for every $\lambda$, it is not clear how we can get a good forking notion for models of different sizes. This is one direction further work could focus on.

Another (non-orthogonal) direction would be to find applications for such a forking notion. As mentioned in the previous section, we believe it could be useful in proving categoricity transfer theorems. Moreover, the frame built in Section \ref{good-frame-without-symmetry} is only well-behaved for $\mu^+$-saturated models, and it would be interesting to know when the class of $\mu^+$-saturated models is an AEC. This calls for tools to deal with unions of saturated models and we plan to explore this further in future work\footnote{Since this paper was first submitted, several extensions have been written. In \cite{indep-aec-v4}, the argument here is axiomatized, the cofinality assumption on the categoricity cardinal is removed and a global independence relation (for types of all lengths) is built (assuming more hypotheses). This is used to prove an approximation to Shelah's categoricity conjecture. In \cite{bv-sat-v3}, it is shown that it follows from $\aleph_0$-local character of splitting and tameness that, for all high-enough cardinals $\lambda$, the union of a chain of $\lambda$-model-homogeneous models is $\lambda$-model-homogeneous. All these works ultimately rely on the methods of this paper.}.

\bibliographystyle{amsalpha}
\bibliography{building-a-good-frame-from-tameness}

\end{document}